\documentclass[11pt]{article}
\usepackage{bbm}
\usepackage{amsmath,amsthm,amssymb}
\usepackage{graphicx}
\usepackage{tikz}
\usepackage[utf8]{inputenc}
\usepackage{pgfplots}

\usepackage{lscape}
\usepackage{algorithmic}
\usepackage{algorithm}
\usepackage{graphicx}
\usepackage{graphics}
\usepackage{subcaption}
\usepackage{booktabs}
\usepackage{xcolor}
\usepackage{multicol}
\usepackage{multirow}
\usepackage{setspace}
\usepackage{pict2e,color}
\usepackage{multirow}
\usepackage{booktabs}
\usepackage{amsfonts}
\usepackage{verbatim} 
\usepackage{sectsty}
\usepackage{url}

\usepackage{empheq}
\usepackage[normalem]{ulem}
\usepackage[titletoc,title]{appendix}
\usepackage[flushleft]{threeparttable}
\usepackage{soul} 
\usepackage{natbib}
\usepackage{algorithm,algorithmic}
\usepackage{etoolbox}
\preto\subequations{\ifhmode\unskip\fi}
\usepackage{booktabs}
\usepackage{multirow}
\theoremstyle{definition}
\newtheorem{definition}{Definition}
\newtheorem{theorem}{Theorem}

\newtheorem{proposition}{Proposition}

\newtheorem{remark}{Remark}

\title{\Large Benders Cut Classification via Support Vector Machines for Solving Two-stage Stochastic Programs}

\begin{document}

\allowdisplaybreaks
\author{Huiwen Jia\thanks{Department of Industrial and Operations Engineering, University of Michigan at Ann Arbor, USA. Email: {\tt hwjia@umich.edu};}
    ~~~Siqian Shen\thanks{Corresponding author; Department of Industrial and Operations Engineering, University of Michigan at Ann Arbor, USA. Email: {\tt siqian@umich.edu}.}}
\date{}

\maketitle

\begin{abstract}
We consider Benders decomposition for solving two-stage stochastic programs with complete recourse based on finite samples of the uncertain parameters. We define the Benders cuts binding at the final optimal solution or the ones significantly improving bounds over iterations as \emph{valuable cuts}. We propose a learning-enhanced Benders decomposition (LearnBD) algorithm, which adds a cut classification step in each iteration to selectively generate cuts that are more likely to be valuable cuts. The LearnBD algorithm includes  two phases: (i) sampling cuts and collecting information from training problems and (ii) solving testing problems with a support vector machines (SVM) cut classifier. We run the LearnBD algorithm on instances of capacitated facility location  and multi-commodity network design under uncertain demand. Our results show that SVM cut classifier works effectively for identifying valuable cuts, and the LearnBD algorithm reduces the total solving time of all instances for different problems with various sizes and complexities.
\end{abstract}
   
\textbf{Keywords}: Benders decomposition, two-stage stochastic (integer) programming, support vector machine (SVM), cut classification

\section{Introduction}
\label{sec:intro} 

In this paper, we focus on the Benders decomposition \citep{benders1962partitioning} and its 
implementation for solving a broad class of two-stage stochastic programming models.
 In the first stage, the value of decision variable $x \in \mathbb{R}^{n_1}$ or $x \in \mathbb{Z}^{n_1}$ is chosen from a feasible region $\mathcal{X}$ before the realization of the uncertainty given the cost vector $c \in \mathbb{R}^{n_1}$.
In the second stage, decision variable $y \in \mathbb{R}^{n_2}$ is  a continuous recourse decision. 
The matrix $W \in \mathbb{R}^{m_2 \times n_2}$, vector $h \in \mathbb{R}^{m_2}$, matrix $T \in  \mathbb{R}^{m_2 \times n_1}$ and cost vector $q \in \mathbb{R}^{n_2}$  
are subject to uncertainty. We denote the overall uncertain parameter as $\xi = [W, h, T,q]$.
A two-stage stochastic programming model is given by
\begin{equation}\label{eq:2-state}
\begin{aligned}
 \  \ & \min_{x \in \mathcal{X}} \ \  c^Tx + \mathbb{E}_{\xi}[ Q(x,\xi) ]
\end{aligned}
\end{equation}
where $\mathbb{E}_{\xi}[\cdot]$ takes expectation of $\cdot$ based on the probability distribution of $\xi$ and
\begin{equation}
\begin{aligned}
 \ Q(x,\xi)  \overset{\textrm{def}}{=}  \ &  \min_y \ \  q^T y \\
& \ \textrm{s.t.} \ \ \ Wy = h - Tx .\\
\end{aligned}
\end{equation}

We consider a finite number of realizations of the uncertain parameter $\xi$, called ``scenarios" or ``samples" in the stochastic programming literature \citep[see, e.g.,][]{birge2011introduction}.
Let $\Omega$ be the sample space that contains all the scenarios and each scenario $\omega \in \Omega$ is associated with a specific realization $\xi_{\omega}=[W_{\omega}, h_{\omega} ,  T_{\omega} ,  q_{\omega}] $ of the uncertain parameter $\xi$. Denote the occurrence probability of scenario $\omega$ as $p_{\omega}$, and thus  $\sum_{\omega \in \Omega} p_\omega = 1$. Model~\eqref{eq:2-state} can be reformulated as
\begin{equation} \label{eq:1}
\begin{aligned}
 \  \ & \min_{x \in \mathcal{X}} \ \  c^Tx + \sum_{\omega \in \Omega} p_{\omega} Q_{\omega}(x) , \\
\end{aligned}
\end{equation}
where
\begin{equation} \label{eq:2}
\begin{aligned}
 \ Q_{\omega}(x) \overset{\textrm{def}}{=} Q(x, \xi_{\omega})  = \ &  \min_y \ \  q_{\omega}^T y \\
& \ \textrm{s.t.} \ \ \ W_{\omega}y = h_{\omega} - T_{\omega}x .\\
\end{aligned}
\end{equation}

Note that the assumption of having finite scenarios is made without loss of generality.  If the uncertain parameter $\xi=[W, h, T, q]$ follows a continuous distribution, one can apply the Monte Carlo sampling approach to generate $N_s$ i.i.d. samples $\{ \omega_1, \ldots, \omega_{N_s} \}$ of the uncertain parameters. 
The second-stage objective function $\mathbb{E}_{\xi}[ Q(x,\xi) ]$ can be replaced by the sample average approximation (SAA) $\frac{1}{N_s}\sum_{i=1}^{N_s} Q_{\omega_i}(x)$ \citep{kleywegt2002sample}. 

\subsection{An Overview of Benders Decomposition}
\label{subsec:bd}

With large $|\Omega|$, Model~\eqref{eq:1} is in general computationally intractable when it involves integer variable $x$. The Benders decomposition, which takes advantage of the decomposable structure of two-stage stochastic programs, is applied widely to optimize variants of Model \eqref{eq:1} formulated for a wide range of applications \citep[see, e.g.,][]{magnanti1981accelerating}.
Creating new variables $\theta_{\omega} \in \mathbb{R}, \ \forall \omega \in \Omega$ in the first-stage problem, one can formulate a relaxation of the original problem, called relaxed master problem (RMP), which has an initial form:
\begin{equation} \label{eq:rmp0}
\begin{aligned}
(\textrm{RMP}^0) \ \ & \min_{x \in \mathcal{X}, \theta} \ \ c^Tx + \sum_{\omega \in \Omega} p_{\omega} \theta_{\omega}.
\end{aligned}
\end{equation}
Subproblems (SPs) are defined as the linear programming dual of the second-stage problems \eqref{eq:2} with dual variable  $\pi_{\omega} \in  \mathbb{R}^{m_2}, \ \forall  \omega \in \Omega $.
We refer to $\textrm{SP}_{\omega}$ as the SP for scenario $\omega$, formulated as
\begin{equation} \label{eq:3}
\begin{aligned}
    (\textrm{SP}_{\omega}) \ \ \ Q^D_{\omega}(x) \ =  \  &  \underset{ \pi_{\omega} }{\max} \ \  (h_{\omega} - T_{\omega}x)^T \pi_{\omega} \\
& \ \textrm{s.t.} \ \ \  W_{\omega}^T \pi_{\omega} \leq q_{\omega} .\\
\end{aligned}
\end{equation}
Let $V^{\omega,t}$ be the set of identified extreme points of the feasible region of $\textrm{SP}_{\omega}$ in iteration $t$, and we have $V^{\omega,0} = \emptyset$.
Similarly, let $R^{\omega,t}$ be the set of identified extreme rays of the feasible region of $\textrm{SP}_{\omega}$ in iteration $t$, and $R^{\omega,0} = \emptyset$.
The two sets are respectively associated with Benders optimality cuts and feasibility cuts generated during iterations $1, \ldots, t-1$, and we will explain the cuts in detail later.
In iteration $t$, the corresponding RMP is given by:
\begin{equation} \label{eq:rmpt}
\begin{aligned}
(\textrm{RMP}^t) \ \   & \min_{x \in \mathcal{X},\ \theta} \ \ c^Tx + \sum_{\omega \in \Omega} p_{\omega} \theta_{\omega} & \\
& \ \textrm{s.t.}   \ \ \   \theta_{\omega} \geq (h_{\omega} - T_{\omega}x)^T \nu_{\omega} & \ \ \omega \in \Omega, \ & \ \nu_{\omega} \in V^{\omega,t};\\
&  \ \ \ \ \ \ \ \  (h_{\omega} - T_{\omega}x)^T \rho_{\omega} \leq 0 & \ \ \omega \in \Omega, \ & \ \rho_{\omega} \in R^{\omega,t}.
\end{aligned}
\end{equation}

After solving $\textrm{RMP}^t$, we obtain an optimal solution $(\hat{x}^t, \hat{\theta}^t)$. 
Then for each scenario $\omega \in \Omega$ and its subproblem $\textrm{SP}_{\omega}$, we first check whether the current $\textrm{RMP}^t$ solution leads to a feasible second-stage problem by solving a corresponding subproblem with decision variable $\sigma_{\omega}  \in  \mathbb{R}^{m_2} $, modeled as 
\begin{equation} \label{eq:4}
\begin{aligned}
(\textrm{SP}_{\omega}\textrm{-}\textrm{F}) \ \   \  &  \underset{ \sigma_{\omega} }{\max} \ & \ (h_{\omega} - T_{\omega}\hat{x}^t )^T \sigma_{\omega} \\
& \ \textrm{s.t.} \ &  \  W_{\omega}^T \sigma_{\omega} \leq \mathbf{0};\\
& & \ \| \sigma_{\omega} \| \leq 1 .\\
\end{aligned}
\end{equation}
If $\textrm{SP}_{\omega}\textrm{-}\textrm{F}$ has a positive optimal objective value with an optimal solution $\bar{\sigma}_{\omega}$, then from any feasible solution to $\textrm{SP}_{\omega}$, we can move along the direction $\bar{\sigma}_{\omega}$ to stay feasible (due to the first constraint of model \eqref{eq:4}) but increase the objective value of $\textrm{SP}_{\omega}$. This implies that $\textrm{SP}_{\omega}$ is unbounded and the second-stage problem is infeasible for given $\hat{x}^t$. 
To cut off the infeasible first-stage solution $\hat{x}^t$, we generate a Benders feasibility cut:
\begin{equation}\label{eq:feasiblecut}
(h_{\omega} - T_{\omega}x )^T \bar{\sigma}_{\omega} \leq 0 
\end{equation}
to $\textrm{RMP}^{t+1}$, which is equivalent to letting $ R^{\omega,t+1} = R^{\omega,t} \cup \{ \bar{\sigma}_{\omega} \}$.
In this paper, we only focus on the case having complete recourse, under which any feasible first-stage solution will result in a feasible second-stage problem.
Therefore, our $\textrm{RMP}^{t}$ (i.e., Model \eqref{eq:rmpt}) for each iteration $t$ only contains the first set of optimality cuts, whose derivation is given as follows. 
If the subproblem is feasible, then in iteration $t$ we check the optimality of $(\hat{x}^t, \hat{\theta}^t)$.
We solve $\textrm{SP}_{\omega}$ with $x = \hat{x}^t$ to obtain an optimal solution $\bar{\pi}_{\omega}$ and the optimal objective value $Q^D_{\omega}(x) = \bar{\pi}_{\omega}^T (h_{\omega} - T_{\omega}\hat{x}^t)$. 
By strong duality, for any value of $x$, $ Q^D_{\omega}(x)  = Q_{\omega}(x)$.
The solution to $\textrm{RMP}^{t}$ will reach the same objective value as the original problem when $\hat{\theta}_{\omega}^t \geq Q_{\omega}(\hat{x}^t), \ \forall \omega \in \Omega$. 
Therefore, $\hat{\theta}_{\omega}^t < Q^D_{\omega}(x)$ indicates that the current solution $(\hat{x}^t, \hat{\theta}^t)$ is not optimal for the original problem.
Thus, we add a Benders optimality cut 
\begin{equation}\label{eq:optcut}
\theta_{\omega} \geq (\bar{\pi}_{\omega})^T(h_{\omega} - T_{\omega}x)  
\end{equation} 
to $\textrm{RMP}^{t+1}$, which is equivalent to letting $ V^{\omega,t+1} = V^{\omega,t} \cup \{ \bar{\pi}_{\omega} \}$. We refer to the cuts being added and the corresponding dual extreme points being identified interchangeably in this paper.

In iteration $t$, the objective value of $\textrm{RMP}^t$ provides a valid lower bound to the origin problem because it is a relaxation. 
If all SPs have finite optimal objective values, $\hat{x}^t$ and all recourse solutions together form a feasible solution to the original two-stage problem, and thus
\begin{equation}
c^T \hat{x}^t + \underset{\omega \in \Omega}{\sum} p_{\omega}Q_{\omega}(\hat{x}^t)
\end{equation}
provides a valid upper bound.
The algorithm terminates when the upper and lower bounds are equal or their gap is within a pre-specified tolerance $\delta$. In this paper, we define the gap as
\begin{equation}\label{eq:relativedif}
\textrm{optimality gap} = \frac{\textrm{upper bound} - \textrm{lower bound}}{ \textrm{lower bound}} \times 100\%.
\end{equation}
The Benders decomposition converges in a finite number of iterations due to the finite number of dual extreme rays and extreme points of the finitely many SPs.

\subsection{Challenges and Research Overview}\label{subsec:research}

While the Benders decomposition method helps to solve two-stage stochastic programs efficiently, it could suffer from slow convergence.
One reason is that the size of RMPs becomes too large due to the quickly increased number of newly added cuts over iterations.
\citet{geoffrion1974multicommodity} are among the first to notice and emphasize on the computational difficulty of solving RMPs for stochastic binary integer programs. \citet{magnanti1981accelerating} report that over  $90\%$  of the total time of implementing the Benders decomposition is spent on solving RMPs.
\citet{minoux1986mathematical} points out that not all extreme points of the feasible region of SPs equally contribute to restricting the optimal solution to RMPs.
Therefore, a larger number of Benders cuts are not tight at the final optimal solution, but can increase the size of RMPs, which are then extremely hard to solve as large-scale integer programs. 

We propose a two-phase learning-enhanced Benders decomposition (LearnBD) algorithm to solve two-stage stochastic integer programs with finite samples of the uncertain parameter and complete recourse, where the second-stage subproblems are  linear programs (LPs).
We define cuts as \emph{valuable cuts} when they can either cut the feasible region in the current iteration significantly, or be tight at the final optimal solution \citep[see][for a similar definition in the latter case]{holmberg1990convergence}.
Our goal is to only add valuable cuts to the corresponding RMP in each iteration.
Up to date, there is no practical and systematic way to perform cut classification and to accelerate the iterative process for Benders decomposition for large-scale optimization problems, according to \citet{rahmaniani2017benders}.
We propose to integrate machine learning techniques into the traditional Benders decomposition framework to learn cut characteristics and selectively generate subsets of Benders cuts iteratively.

\subsection{Contributions of the Paper}\label{subsec:contribution}

We summarize the main contributions of this paper as follows.
\begin{itemize}
    \item Firstly, we identify a set of characteristics and quantify performance measures of Benders cuts. We construct a cut classifier using support vector machines (SVM), a widely used supervised machine learning method that takes history observations and their labels as input, to identify valuable cuts in each iteration. 
    \item Secondly, we develop the LearnBD algorithm with SVM cut classifier, to limit the size of RMPs and reduce total solving time.  We also provide guidelines for choosing hyperparameters for enhancing the effectiveness of the LearnBD algorithm.
    \item Thirdly, we test instances of capacitated facility location  and multi-commodity network design under uncertain demand, to demonstrate the computational advantages of LearnBD in different problem settings. 
    Our results show that the LearnBD algorithm leads to smaller sizes of RMPs with fewer accumulated cuts, and therefore shorter time for solving RMPs. 
\end{itemize}

\subsection{Structure of the Paper}\label{subsec:paperstructure}

The remainder of this paper is organized as follows. 
In Section~\ref{sec:lr}, we review the literature on the effort of improving Benders decomposition for solving large-scale optimization problems.
In Section~\ref{sec:learning-enhanced}, we develop the LearnBD algorithm and use SVM for constructing the cut classifier. 
In Section~\ref{sec:compu}, we present the computational results of the LearnBD algorithm benchmarked with the traditional Benders approach. 
In Section~\ref{sec:conclusion}, we conclude the paper and describe future research.

\section{Literature Review}\label{sec:lr}  

The Benders decomposition was initially proposed by \citet{benders1962partitioning} and was then widely used for solving problems of scheduling and planning \citep{cordeau2001benders,hooker2007planning}, network  optimization and transportation \citep{laporte1994priori,costa2005survey,binato2001new}, and inventory control and management \citep{federgruen1984combined,cai2001solving}.
\citet{magnanti1981accelerating} and \citet{naoum2013interior} note that directly applying the traditional Benders decomposition may require excessive computational effort. It is mainly due to the poor convergence of RMPs that has been computationally demonstrated in \citet{orchard1968advanced} and \citet{wolfe1970convergence}. Several researchers have proposed enhancement strategies depending on different problem structures to accelerate the algorithm accordingly, of which we describe the details below.

In the traditional Benders approach detailed in Section~\ref{subsec:bd}, RMPs and SPs are solved iteratively, and thus the first stream of studies concentrates on problem-solving techniques, and particularly techniques for efficiently computing RMPs or SPs. \citet{geoffrion1974multicommodity} propose to only sub-optimally solve RMPs in each iteration to enable cut generation, without seeking tight cuts at the beginning of the Benders approach.
Similarly, \citet{raidl2015decomposition} solves RMPs using heuristics to save computational time. 
\citet{zakeri2000inexact} show that sub-optimal solutions to the SPs can still generate valid cuts in RMPs, and thus effective heuristic approaches are designed for solving the SPs approximately.

The second stream of studies focuses on decomposition strategies, to guide the process of partitioning variables to remain in RMPs or in SPs. 
\citet{crainic2014partial} 
propose a partial Benders decomposition algorithm to reduce the number of feasibility and optimality cuts, while adding information of SPs into RMPs by retaining or creating scenarios. 
They develop different decomposition strategies for choosing the retaining scenarios when solving two-stage mixed-integer programming (MIP) models with continuous recourse.
In addition, \citet{gendron2016branch} propose a non-standard decomposition strategy, which retains the second-stage variables in RMPs, and the authors test the results using instances of network design problems. 

Machine learning techniques have been applied to general-purpose optimization algorithms for urging quick convergence to optimal or sub-optimal solutions \citep[see, e.g.,][]{he2014learning,khalil2017learning}. 
\citet{khalil2016learning} introduce a novel data-driven framework for variable selection to solve MIP models via branch-and-bound algorithm efficiently.
\citet{kruber2017learning} develop a supervised learning approach to distinguish a stronger reformulation of a given MIP model and to determine which decomposition to implement in order to improve the speed of MIP solvers.
\citet{misra2018learning} directly construct a model for seeking the optimal solution as a function
of the input parameter, by learning relevant sets of active constraints given computationally expensive large-scale
parametric models.
 
Recently, several papers apply machine learning to improve algorithmic efficiency of decomposition approaches, especially focusing on cut classification. Among them, \citet{rl-tang2019reinforcement} model the cut selection in integer programming as a reinforcement learning (RL) problem. They define the corresponding concepts in RL and implement an offline training phase. \citet{sl-baltean2019scoring} develop linear outer-approximations of semi-definite constraints that can be effectively integrated into global solvers. They construct a neural network for predicting the objective improvement of each cut, which is similar to the performance measure proposed in our paper.

\section{Learning-enhanced Benders Decomposition}\label{sec:learning-enhanced} 

 We develop a Learning-enhanced Benders Decomposition (LearnBD) for solving two-stage stochastic programs with complete and continuous recourse in the second stage. 
The algorithm aims to solve a set of two-stage problems that share similar problem structures (i.e., dimensions of decision variables, constraint matrices, and cost parameters in the objective function) but could have different realizations of the uncertain parameter. As a result, LearnBD constructs a training problem that shares the same problem structures as the original problem(s), while the distributions of the uncertain parameter can be different.
LearnBD samples cut and collects information from the training problem. 
Then, it uses the collected cut information to train an SVM cut classifier and then optimizes the original problem(s). 

To find potential applications of LearnBD, consider some industries where we need to periodically solve similar optimization problems with the same system structures and decision frameworks, but with different input data representing the current environment and status of the system.
For example, the unit commitment problem in the power system is solved every hour to determine the operational schedule of the generating units under random renewable generation and electricity loads \citep[see, e.g.,][]{saravanan2013solution,dashti2016weekly}. A grid  operator needs to solve a stochastic program in the form of \eqref{eq:1} every hour with different $\xi_{\omega}$-values. If one can solve these similar problems in an efficient way, it can significantly improve the operational efficiency of power grids. The proposed LearnBD can be applied in this case, where one can sample cuts from training problems using  previous days' data , and re-use the cut classifier for the stochastic programs to be solved in future hours. 
Moreover, even for the problems that we only solve once, LearnBD could be useful.
For instance, consider solving stochastic programming models using SAA, where we can solve a number of SAA-based reformulations with different i.i.d. samples repeatedly, by using cut information collected from solving one such reformulation as a training problem. 

LearnBD includes two phases: Offline cut sampling (Phase 1) and solving a given problem using cut classification (Phase 2). The collected cut information can be used to solve any testing problem of the same variable-and-constraint size. Therefore, the time spent on Phase 1 does not affect the total solving time. In Phase 1, we solve training problems under the estimation of the uncertainty and collect training data for cut classification. 
In Phase 2, for a given testing problem, which is viewed as an unseen testing instance, we train cut classifiers using the training data from Phase 1 and apply cut classification steps throughout the Benders iterations.

We provide an overview of LearnBD in Figure~\ref{figure:3-1}, in which 
related to Phase 1, $K$ is the number of sampling paths, $N$ is the length of each sampling path, and $\textrm{RMP}^n_k$, $n=1,\ldots,N,$ $k=1,\ldots,K$ is the RMP of the training problem corresponding to iteration $n$ in sampling path $k$.
Related to Phase 2, $\textrm{RMP}^t, \ t = 0, \ldots, T$ is the  RMP of the testing problem corresponding to  iteration $t$ and $t=T$ denotes the last iteration.
 In Phase 1, we perform cut sampling to generate training data, including cut characteristics and performance measures (see Section~\ref{subsec:cut-sampling}). In Phase 2, we utilize the classifier to distinguish valuable cuts from all generated cuts in each iteration and solve RMPs iteratively by only adding valuable cuts (see Section~\ref{subsec:functionapply}). 
As a sub-procedure in Phase 2, we train an SVM classifier with the training data generated in Phase 1, which takes cut characteristics as input and $\{1,-1\}$ valued label as output to classify whether or not a cut is valuable (see Section~\ref{subsec:classifierconstruct}).

\begin{figure}[htbp]
    \centering
    \includegraphics[width=0.9\textwidth]{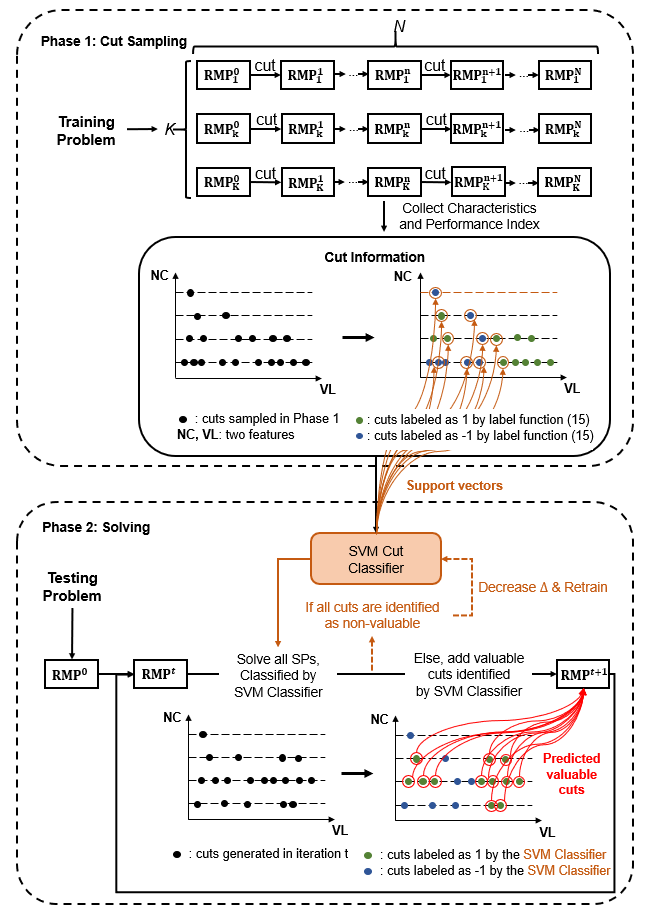}
\caption{An overview of LearnBD procedures. In Phase 1, we collect cuts from training problems and label them as valuable cuts (green points) and non-valuable cuts (blue points). Then, we train an SVM classifier and use it to classify the new cuts we meet in Phase 2 when solving the testing problem. In Phase 2, we only include the valuable cuts identified by the SVM classifier in each iteration.}
    \label{figure:3-1}
\end{figure}


\subsection{Phase 1: Cut Sampling}\label{subsec:cut-sampling} 

In Phase 1, we conduct cut sampling from some training problem to collect the information of valuable cuts, which will then be used to train the classifier in Phase 2.
The training data set $D$ can be viewed as a $D_{row} \times D_{col}$ matrix, where each row is the information of a specific sample cut, the first $D_{col} - 1$ columns are cut characteristics, and the last column is a $\{-1,1\}$ valued label.

\paragraph{Characteristics.}
Cut characteristics are features of a cut that can help us predict the performance of the cut in future iterations if it is added to the current RMP.
We consider the following two characteristics.
The first is cut violation at the current solution $(\hat{x}^t, \hat{\theta}^t)$ of $\textrm{RMP}^t$, denoted by VL and it can be computed as $\pi_{\omega}^T (h_{\omega} - T_{\omega} \hat{x}^t) - \hat{\theta}^t_{\omega} $, according to \eqref{eq:optcut}. 
This characteristic reflects how large the feasible region of $\textrm{RMP}^t$ can be cut off if adding the cut.
The second characteristic is related to the scenario where a cut is generated from.
We denote the number of cuts generated by the same scenario in previous iterations as NC. 
This characteristic reflects the trade-off between exploration and exploitation, two typical learning strategies. 
A preference to a cut whose associated scenario generates more cuts in previous iterations, links to an exploitation strategy, while the opposite preference leads to an exploration strategy. 
On the one hand, a large number of cuts generated from the same scenario shows that this scenario is crucial for identifying an optimal solution. 
However, it could be the case where the majority of valuable cuts in this scenario have already been generated. Thus, the change of the objective value of RMP brought by a new cut from the same scenario can be small. 
Therefore, the relationship between NC and future performance of a cut is highly possible to be nonlinear. 
A collection of characteristics of one specific cut is referred to as an observation $\mathbf{o}$, with $\mathbf{o} = (\textrm{VL}, \textrm{NC})$.

\paragraph{Performance index.}
In the training data set, each observation also needs to be assigned a label $l$, where $1$ is assigned to valuable cuts and $-1$ is assigned to non-valuable cuts.
Therefore, we define a performance index of each cut and then transform it into $\{-1,1\}$ valued label. 
We choose the change amount of the objective value of $\textrm{RMP}^t$ before and after adding a cut as the performance index of the cut, denoted by PI. 
We add exactly one cut to $\textrm{RMP}^t$ each time to recognize the change of objective value brought by the cut. 
In practice, users can customize the characteristics and performance index according to specific applications.
The rule for transforming the performance index will be discussed after we introduce sampling paths next.

\paragraph{Sampling path.}\label{para:sp}
We construct sampling paths to guide the cut sampling process to record the cut characteristics and performance index.
The number of sampling paths $K$ and the length of sampling path $N$ are pre-determined hyperparameters.
In each sampling path $k, \ k = 1, \ldots, K$, we start with $\textrm{RMP}^0_k$, which is initialized by $\textrm{RMP}^0$.
In iteration $n$ of a sampling path $k$, for $ n = 0, \ldots, N-1, \ k = 1, \ldots, K$, we solve $\textrm{RMP}^n_k$ and obtain an optimal solution $(\hat{x}^n_k, \hat{\theta}^n_k)$.
Then we follow the Monte Carlo sampling approach to randomly sample one scenario $\omega \in \Omega$ and solve the corresponding $\textrm{SP}_{\omega}$ by plugging in $(\hat{x}^n_k, \hat{\theta}^n_k)$:
(i) if no optimality cut is generated, then we continue sampling another scenario and solving the corresponding SP; 
(ii) if an optimality cut is generated, we record the two characteristics, instantly add the cut to $\textrm{RMP}^{n+1}_k$, and then record the performance index.
Similar to Section~\ref{subsec:bd}, we use $V^{\omega, n}_k$ to denote the set of identified extreme points of the feasible region of $\textrm{SP}_{\omega}$ in iteration $n$ of sampling path $k$, and $\textrm{RMP}^n_k$ is defined in the following form:
\begin{equation} \label{eq:rmp_k^n}
\begin{aligned}
(\textrm{RMP}^n_k) \ \   & \min_{x \in \mathcal{X},\ \theta} \ \ c^Tx + \sum_{\omega \in \Omega} p_{\omega} \theta_{\omega} & \\
& \ \textrm{s.t.}   \ \ \   \theta_{\omega} \geq (h_{\omega} - T_{\omega}x)^T \nu_{\omega} & \ \ \omega \in \Omega, \ & \ \nu_{\omega} \in V^{\omega,n}_k.\\
\end{aligned}
\end{equation}
Once a new cut is generated, we move one step forward in one sampling path and therefore the iterative process stops after reaching $\textrm{RMP}^N_k$ in sampling path $k, \ k = 1, \ldots, K$.

 Through cut sampling, we collect $ \Gamma = N \times K$ number of training data. A larger set of training data in general leads to a more precise classifier.
Larger $N$ means that we can collect information of more representative cuts because we solve RMPs in a wider range of problem sizes. 
Different independent sampling paths can be conducted in parallel, and therefore, larger $K$ will not significantly increase the time of Phase 1.
However, the cuts generated by RMPs with similar sizes can share similar characteristics.
These similar inputs can also lead to over-fitting and can eventually weaken the power of the classifier.
In our later computational studies, we choose $N=2 \times |\Omega|$ and $K=2$.

\begin{remark}\label{remark:indpath}
The cut sampling process is independent across all sample paths, and thus the cut information collected in different sampling paths is independent of one another.
\end{remark}

\begin{remark}\label{remark:phase 1 time}
These sampled cut information can be re-used in Phase 2 when solving different testing problems and thus the time of Phase 1 does not affect the total solving time of testing problems.
\end{remark}

\begin{algorithm}
 \caption{Phase 1 of the LearnBD algorithm. }
 \label{alg:learn1}
 \begin{algorithmic}[1] 
 \STATE \textbf{Input:}
            a two-stage stochastic program with a set $\Omega$ of scenarios; values of  $N$, $K$, $\Delta$.
 \STATE \textbf{Initialize:}  
            $\textrm{RMP}^0$ and 
            $\textrm{SP}_{\omega}$, $\forall \omega \in \Omega$  of the training problem,
            $D  \leftarrow \emptyset$.
 \FOR {$k = 1, \ldots, K$}   \label{Line:2}
     \STATE \textbf{Initialize}:  
            $V^{\omega,0}_k \leftarrow  \emptyset, \ \forall \omega \in \Omega$, 
            $\textrm{NC}_{\omega} = 0, \ \forall \omega \in \Omega$,
            $\textrm{RMP}^0_k \leftarrow \textrm{RMP}^0$ ,
            $D_{\textrm{temp} } \leftarrow \emptyset$.
     \STATE Solve $\textrm{RMP}^{0}_k$, obtain an optimal solution $\{ \hat{x}^0_k, \hat{\theta}^0_k \}$ with optimal objective value $\hat{z}^0_k$.
     \FOR { $n = 0, \ldots, N-1$} 
          
         \STATE\label{Line:7} Randomly select $\omega' \in \Omega$, solve $\textrm{SP}_{\omega'}$, obtain an optimal solution $ \pi_{\omega'}$ and its objective value $\zeta_{\omega'}$.
         \IF { $(\hat{\theta}^n_{k})_{\omega'} < \zeta_{\omega'}$} 
                \STATE  $V^{\omega,n+1}_k \leftarrow V^{\omega,n}_k \cup \{ \pi_{\omega'} \}$ , $\textrm{NC}_{\omega'} \leftarrow  \textrm{NC}_{\omega'} + 1$, 
                        $\textrm{VL} \leftarrow  \zeta_{\omega'} - (\hat{\theta}^n_{k})_\omega $;
                \ELSE  \STATE Go to Step~\ref{Line:7}.
                \ENDIF 
         \STATE\label{Line:13} Solve $\textrm{RMP}^{n+1}_k$, obtain an optimal solution $\{ \hat{x}^{n+1}_k, \hat{\theta}^{n+1}_k \}$ with optimal objective value $\hat{z}^{n+1}_k$, 
                $\textrm{PI} \leftarrow  |\hat{z}^{n+1}_k - \hat{z}^{n}_k| $,
         \STATE\label{Line:14} $(D_{\textrm{temp}})_{n+1,.} \leftarrow ( \textrm{VL}, \textrm{NC}_{\omega'}, \textrm{PI}  )$ .
     \ENDFOR
     \STATE $D \leftarrow D \cup ( (D_{\textrm{temp}})_{N,1}, (D_{\textrm{temp}})_{N,2}, 1)$
    \FOR { $n = N-1, \ldots, 1$}
        \IF {$ \frac{(D_{\textrm{temp}})_{n,3}}{(D_{\textrm{temp}})_{n+1,3}}  < \Delta$}
            \STATE $D \leftarrow D \cup ( (D_{\textrm{temp}})_{n,1}, (D_{\textrm{temp}})_{n,2}, -1)$
        \ELSE  \STATE $D \leftarrow D \cup ( (D_{\textrm{temp}})_{n,1}, (D_{\textrm{temp}})_{n,2}, 1)$
        \ENDIF
    \ENDFOR
 \ENDFOR
 \RETURN $D$
 \end{algorithmic}
 \end{algorithm}

\subsection{Subroutine in Phase 2: Classifier Construction}\label{subsec:classifierconstruct} 

 We introduce a subroutine in Phase 2, i.e., constructing SVM classifiers with training data $D$, to predict the potential performance of cuts and identify valuable cuts. 
As mentioned in Section~\ref{subsec:cut-sampling}, $D$ can be presented as a collection of observations and labels of sample cuts, where $D = \{ (\mathbf{o}_d, l_d) , \ d = 1, \ldots, \Gamma \}$ and more specifically each cut observation $\mathbf{o} = (\textrm{VL}, \textrm{NC})$.
In Section~\ref{subsubsec:svm}, we show how SVM works and how to estimate the parameters with training data. 
In Section~\ref{subsubsec:svmad}, we discuss the advantages of SVM as a cut classifier.
 
\subsubsection{Building Cut Classifier Using Support Vector Machines (SVM).}\label{subsubsec:svm} 

SVM is a well-known supervised machine learning approach \citep[see][]{cortes1995support,vapnik1998statistical,vapnik1999overview,vapnik2013nature} and has been used for analyzing data in many applications. 
Given a training data set $D = \{(\mathbf{o}_d, l_d), \ d = 1, \ldots, \Gamma \}$ from Phase 1, where $\mathbf{o}_d \in \mathbb{R}^\Sigma$ is an observation (in our problem $\Sigma = 2$), a subset of training data are identified as support vectors  after the training process.
Parameterized by a coefficient vector $\mathbf{a} \in \mathbb{R}^{\Gamma}$, and an intercept $b \in \mathbb{R}$, the SVM classifier $f_{SVM}(\cdot): \mathbb{R}^{\Sigma} \rightarrow \{-1,1\}$ for a new observation $\mathbf{o}' \in \mathbb{R}^{\Sigma} $ (i.e., the collected information of a specific cut in our problem), is given by
\begin{equation}\label{eq:predict}
f_{SVM}(\mathbf{o'}) = \textrm{sign} \bigg[\sum_{d=1}^{\Gamma} l_d  a_d K(\mathbf{o'}, \mathbf{o}_d) + b \bigg],
\end{equation}
where $K(\cdot,\cdot): \ \mathbb{R}^{\Sigma} \times \mathbb{R}^{\Sigma}  \rightarrow \mathbb{R}$ is a predetermined kernel function. 
 Here we use one of the most popular kernel functions, the Radial Basis Function (RBF), in which $K(\mathbf{o}_1, \mathbf{o}_2) = \textrm{exp} (-\gamma \| \mathbf{o}_1 \cdot \mathbf{o}_2 \|^2 )$.

\paragraph{Label transformation.}
We define a label transformation function to transform continuous performance index into $\{ -1, 1 \}$ label, where $1$ indicates a valuable cut. 
 The nature of the convergence of Benders decomposition is accompanied by the fact that the change of the objective value of RMPs is decreasing over iterations.
Therefore, we treat the cuts that can bring a large enough proportion of PI of the next cut in the same sampling path as valuable cuts.
We directly assign label $1$ to the last cut of each cut sampling path.
For other cuts, we calculate the ratio of its PI and the PI of the next cut in the same sampling path and then compare the ratio with a pre-determined threshold $\Delta \in [0,2]$.
The label transformation function is defined as:
\begin{equation}\label{eq:label}
l^n_k = 
\left\{
\begin{aligned}
& -1, \ \ \ & \textrm{if } \frac{\textrm{PI}^n_k}{\textrm{PI}^{n+1}_k} < \Delta  \\
& 1,   & \textrm{otherwise,} \ \ \ \ 
\end{aligned}
\right.
  \ \ n = 0, \ldots, N-1, \ k = 1, \ldots, K
\end{equation}
Larger $\Delta$ shows a more strict rule for recognizing a cut as a valuable cut.
With this label transformation function,  one can calculate all labels using current performance indices and use these labels to train an SVM classifier.  We present the algorithmic details in Algorithm~\ref{alg:learn1}.

\begin{remark}
Label transformation function eliminates a degree of dependency across cuts generated in the same sampling path. Together with Remark \ref{remark:indpath}, all training data are independent with each other.
\end{remark}

The label prediction function $f_{SVM}(\cdot)$ can be interpreted as follows.
We can treat the coefficient $a_d$ as a significance-magnitude of the corresponding data point $d=1,\ldots, \Gamma$, because the label $l_d$ is always shown in $a_d \times l_d$ in the predicting process and $a_d \times l_d$ as a whole indicates the power of data point $d$ for classifying new cuts.
The kernel function $K(\mathbf{o'}, \mathbf{o}_d)$ presents the similarity between the characteristics of a new cut $\mathbf{o}'$ and cut $\mathbf{o}_d$.
Then the label of $\mathbf{o}'$ is the sign of a sum of magnitude-adjusted label of all training data points plus an intercept $b$.
Then by eliminating the training data with zero estimated coefficients $a_d$, the remaining training data form the support vector set $S$, which is a subset of $D$, and function \eqref{eq:predict} can be simplified as
\begin{equation}\label{eq:predict2}
f_{SVM}(\mathbf{o'}) = \textrm{sign} \bigg[\sum_{\mathbf{s} \in S} l_s  a_s K(\mathbf{o'}, \mathbf{o}_s) + b \bigg].
\end{equation}

The parameters $\{ S, \mathbf{a}, b \}$ can be trained by minimizing the prediction loss function among training data as well as maximizing the flatness of the boundary between valuable and non-valuable cuts. 
The prediction loss is computed by the hinge loss function to improve the model sparsity. 
Given the estimated result $u = \sum_{d =1}^{\Gamma} l_d  a_d K(\mathbf{o'}, \mathbf{o}_d) + b$ from Equation \eqref{eq:predict} and the ground truth label $l$, the loss is calculated by:
\begin{equation}
\textrm{Loss}(u,l) = \max (0, \ 1 - l\cdot u ).
\end{equation}
It can be seen that when $l$ and $u$ have the same sign and $|u| \geq 1$, the loss $= 0$; otherwise the loss $= |u - l|$.
 For brevity, we elaborate on the training process of SVM in Appendix \ref{app:svm}.

\begin{remark}
The penalty hyperparameter $C$ balances the explanatory and predictive power of the classifier. 
In general, a larger $C$ shows a smaller tolerance of prediction error within the training dataset and hence results in a classifier with higher explanatory power while too large $C$ will destroy the predictive power. 
The discount rate $\gamma$ in RBF kernel determines the magnitude of similarity between observations, which is related to the model sensitivity and convergence property.
Proper $(\gamma, C)$ will generate a relatively small number of support vectors with an accepted classification accuracy. 
 The classification accuracy is defined as the percentage of the given cuts whose predicted labels are the same as the input labels.
In our later computational studies, the classification accuracy of classifiers on the training data is almost $100\%$.
Those two hyperparameters are generally selected together via cross-validation and grid search to reach the best empirical performance.
 Typically, the larger $C$ we use, the more support vectors can be identified by the classifier, and thus the classification effort will increase. The classification accuracy on the training data can be improved, while the accuracy on unseen testing data can be impaired.

\end{remark}

\subsubsection{Reasons for Choosing SVM.}\label{subsubsec:svmad}  

The advantage of using support vector type of methods for cut classification is threefold. 
Firstly, with the help of the hinge loss function, SVM only selects representative observations from the training data. Those observations are referred to as support vectors and are stored for future classification. This sparse nature increases the computational speed for evaluating new cuts.
Secondly, the mechanism of SVM can be explained by using the similarity between a new cut and all support vectors to predict future performance, which is consistent with our assumptions and motivation that valuable cuts share similarities. Furthermore, the kernel-based method can flexibly help capture the nonlinear relationship between cut performance and characteristics.
Thirdly, the solving process of SVM is a convex optimization problem (see model (SVM-P) in \eqref{eq:svmoptprob}) which is computationally tractable. 

Another support vector type of learning method is support vector regression (SVR) \citep[see, e.g.,][]{Smola2004A}.  The main idea of those two approaches is similar. SVM classifies cuts by $\{1, -1\}$ labels and works as a classifier. SVR evaluates continuous scores of cuts and works as a regressor, which is more informative than a classifier because it can distinguish more rank levels and also allows any fractional rank between levels. We choose SVM over SVR following concerns listed as follows.  Indeed, cuts have different levels of effectiveness for improving RMP solutions. However, we choose not to spend time and effort to fully distinguish between those levels. Recall that the geometric explanation of Benders decomposition is to cut off the feasible region of $\textrm{RMP}^t$ in each iteration $t$. The cuts generated in one iteration can be linearly independent with each other, and thus they cut the feasible region from different directions. Therefore, it is better to include several valuable cuts rather than only one cut in each iteration. On the other hand, since we do not select the cuts with relatively low effectiveness, we do not even need to distinguish among those non-valuable cuts. A similar reason is also mentioned in the learning approach used for a branch-and-bound algorithm by \citet{khalil2016learning}.

\citet{rl-tang2019reinforcement} employ a neural network for selecting cuts for solving integer programming models. 
One advantage of utilizing a neural network is the complex and high-dimensional data it can handle and process. 
The information we use contains coefficients of the current cut and those of all added cuts. If we use a neutral  network, it can evaluate each cut adaptively to the solving process.
In LearnBD, we achieve this iteration-adaptive property by allowing retraining of the cut classifier (see Remark~\ref{remark:retrain}).
The training time of SVM classifier is much shorter than that of neural networks, and the classifier can be trained before solving the testing problems (see Remark~\ref{remark:phase 1 time} and Remark~\ref{remark:compu retraining times}).

\subsection{Phase 2: Cut Classification}\label{subsec:functionapply} 

\paragraph{Iteration rule.}
 In Phase 2, we solve a given two-stage model in the form of \eqref{eq:1}, which is also referred to as the testing problem.
In iteration $t$, we solve $\textrm{RMP}^{t}$ of the testing problem and obtain an optimal solution $(\hat{x}^t, \hat{\theta}^{t})$; 
by plugging in $(\hat{x}^t, \hat{\theta}^{t})$, we solve all $\textrm{SP}_{\omega}, \ \forall \omega \in \Omega$ of the testing problem and record the two characteristics of each generated cut.
Using the characteristic information, the SVM classifier assigns label $1$ to valuable cuts and $-1$ to non-valuable cuts (see Section~\ref{subsec:classifierconstruct}). 
Then we add all valuable cuts with label $= 1$ to $\textrm{RMP}^{t+1}$. 
We use $\overline{V}^{\omega,t}$ to denote the identified extreme points of the feasible region of $\textrm{SP}_{\omega}$ in iteration $t$ in Phase 2.
We repeat adding cuts until the optimality gap between upper bound and lower bound, which is defined in \eqref{eq:relativedif}, is less than a pre-specified tolerance $\delta$. 
If no cut is labeled $1$ by SVM classifier, but we have not reached the optimal tolerance, then we retrain the cut classifier with a smaller $\Delta$ and continue iterating. We define a decreasing list $\mathcal{L}_{\Delta}$ as potential values of $\Delta$, and retrain the SVM classifier by plugging in $\Delta = \mathcal{L}_{\Delta}(l)$ in the $l^{\textrm{th}}$ retraining. 
The algorithmic details of Phase 2 are presented in Algorithm~\ref{alg:learn3}.

\begin{remark}\label{remark:retrain}
As mentioned in Section~\ref{subsubsec:svmad}, we allow retraining to achieve the iteration-adaptive property of Benders decomposition.
If we define the state of the two-stage optimization problem as the set of added cuts to RMP, then the number of possible states is extremely large because the new cuts are also directly affected by the previous solving trajectory.
Thus, it is highly likely that the state we see during the solving process may not have been encountered during the training data collection process and consequently, the relationship between the cut features and valuable labels may not reflect the true relationship. 
In the context of machine learning, this problem, induced by encountering unseen data points, is also defined as \textit{distribution shift}.
When distribution shift happens, the previous classifier does not work anymore for predicting the labels of cuts generated in an unseen state.
Actually, distribution shift may happen in several machine learning algorithms while solving sequential decision-making problems, such as algorithms based on behavioral cloning in Imitation Learning \citep[see, e.g.,][]{didi-kdd-pomerleau1991efficient}, and thus several studies focus on remedying distribution shift \citep[see, e.g.,][]{didi-kdd-ross2011reduction,didi-kdd-reddy2019sqil}.
In this paper, we propose to mitigate distribution shift by retraining the SVM classifier with decreasing $\Delta$ in \eqref{eq:label}, or equivalently, we enforce a less strict standard for valuable cuts in later iterations.
\end{remark}

\begin{remark}\label{remark:train time}
In algorithmic steps, the SVM classifier is retrained several times in Phase 2.
In practical implementation, one can train classifiers with different values of $\Delta$ before starting Phase 2 and call classifiers with the specific $\Delta$-value when solving a problem. Thus, these classifiers can be re-used and the training time of classifiers does not affect the total solving time of testing problems.
 We summarize the numerical performance of retraining in Remark \ref{remark:compu retraining times} in Section~\ref{subsec:results cmnd}.
\end{remark}

 \begin{algorithm}
 \caption{Phase 2 of the LearnBD algorithm. }
 \label{alg:learn3}
 \begin{algorithmic}[1] 
 \STATE \textbf{Input:} $\textrm{RMP}^0$ and $\textrm{SP}_{\omega}, \forall \omega = 1, \ldots, \Omega$ of the testing problem, value of $\delta$.
 \STATE \textbf{Initialize:} set of generated cuts $\overline{V}^{\omega,0} \leftarrow \emptyset$ and number of cuts $ \overline{\textrm{NC}}_{\omega} \leftarrow \emptyset , \ \forall \omega \in \Omega $, list $\mathcal{L}_{\Delta}$, $l=1$, train an SVM classifier with $\Delta = \mathcal{L}_{\Delta}(l)$.
\STATE Solve $\textrm{RMP}^0$, obtain an optimal solution $\{ \hat{x}^0, \hat{\theta}^0 \}$ and optimal objective $\hat{z}^0$, $t\leftarrow 0$, $\textrm{UB}  \leftarrow +\infty$,  LB$ \leftarrow \hat{z}^0$.
\WHILE{ $\frac{\textrm{UB} - \textrm{LB}}{  \textrm{LB}} > \delta$} \label{step:while}
\STATE $n_{cut} \leftarrow 0$.
            \FOR{ $\omega \in \Omega$ } 
                \STATE Solve $\textrm{SP}_{\omega}$, obtain an optimal solution $\pi_{\omega}$ and its optimal objective value $\zeta_{\omega}$.
                \IF { $(\hat{\theta}^t)_\omega < \zeta_{\omega}$} 
                        \STATE\label{line:cla-s}  $\textrm{VL} \leftarrow \zeta_{\omega} - (\hat{\theta}^t)_\omega $.
                        \STATE Input  $  \{ \textrm{VL}, \overline{\textrm{NC}}_{\omega} \}$ into SVM classifier.
                        \IF{ Predicted label is $1$} 
                            \STATE  $\overline{V}^{\omega,t+1} \leftarrow \overline{V}^{\omega,t} \cup \{ \pi_{\omega} \}$; $n_{cut}   \leftarrow n_{cut}   + 1$; $\overline{\textrm{NC}}_{\omega} \leftarrow \overline{\textrm{NC}}_{\omega} + 1$.
                        \ENDIF \label{line:cla-e}
                 \ENDIF 
            \ENDFOR
\IF {$n_{cut} = 0$} \STATE  $l=l + 1$, train an SVM classifier with $\Delta = \mathcal{L}_{\Delta}(l)$; Continue.\ENDIF
\STATE UB $\leftarrow \min \{ \hat{z}^t + \sum_{\omega \in \Omega} \zeta_\omega\}$.
\STATE $t \leftarrow t+ 1$, re-solve $\textrm{RMP}^{t}$, obtain an optimal solution $\{ \hat{x}^{t}, \hat{\theta}^t \}$, optimal objective $\hat{z}^t$.
\STATE $\textrm{LB} \leftarrow \max \{ LB, \hat{z}^t\}$.
 \ENDWHILE
\RETURN Optimal solution $\{ \hat{x}^t, \hat{\theta}^t \}$, optimal objective value $ \hat{z}^t$.
 \end{algorithmic}
 \end{algorithm}

\section{Numerical Studies}\label{sec:compu}

We evaluate LearnBD on two classes of stochastic programs: (i) Capacitated facility location problem (CFLP) and (ii) fixed charge multi-commodity network design problem (CMND). 
CFLP contains binary first-stage variables and continuous second-stage variables with complete recourse.
CMND contains binary first-stage variables and continuous second-stage variables. The traditional formulation of CMND \citep[see, e.g.,][]{crainic2014partial} also involves feasibility cuts in the Benders decomposition. We introduce auxiliary variables in the formulation so that the complete recourse assumption holds. These problems naturally appear in many applications \citep[see, e.g.,][]{melkote2001capacitated,klibi2010design,klibi2012scenario}, and they are notoriously hard to solve \citep[see, e.g.,][]{geoffrion1974multicommodity,birge2011introduction,crainic2001bundle,crainic2011progressive}.

Section~\ref{subsec:instance setup} describes our experimental design and computational settings. Section~\ref{subsec:results svm} shows prediction accuracy of the SVM Classifier over different validation sets.
Section~\ref{subsec:pcstat} presents the overall results of diverse-sized instances and Section~\ref{subsec:pciteration} presents detailed computational results over iterations.
Section~\ref{subsec:training time} presents the time spent on Phase 1 and on training SVM classifiers.
Section~\ref{subsec:transfer} provides results of LearnBD using a classifier trained with cut information collected from another instance.
All the tests are performed on a computer with an Intel Core E5-2630 v4 CPU 2.20 GHz and 128 GB of RAM.

\subsection{Experimental Setup and Test Instances}\label{subsec:instance setup}

\subsubsection{Capacitated Facility Location Problem (CFLP).}

Consider a set $W$ of production plants (facilities) and a set $F$ of factories which have uncertain demand $\tilde{d}$.
The setup cost of facility $i, \ \forall i \in W$ is $k_i$ and the production capacity limit is $u_i$.
The demand of factory $j, \forall j \in F$ is uncertain and can be satisfied by products produced in facility $i, \ \forall i \in W$ if it is open with a unit transportation cost $c_{ij}$, and the unmet demand will generate lost-sale with a unit penalty cost $\rho_j$.
One needs to decide a subset of facilities to open before the realization of the demand to minimize the expected total cost.
We provide the details about RMP and SP formulations in Appendix~\ref{app:cflp}.

For our studies, we use problem sets IV and VI in \citet{beasley1988algorithm}, which are originally from \citet{akinc1977efficient} and \citet{christofides1983extensions}. 
Each problem set uses the same network and identical capacity among facilities. Table~\ref{table:fl-ins} summarizes the attributes of the instances, which are originally proposed for the deterministic capacitated facility location problem, and we apply the techniques in \citet{song2014chance} for sampling scenarios. The demand $\tilde{d}_j$ of factory $j$ in each scenario $\omega$ follows a Normal distribution with mean equal to the demand used in the original deterministic instances and standard deviation equal to $0.1 - 0.2$ times the mean.
\begin{table}[htbp]
  \centering
  \caption{Instance Attributes of CFLP}
    \begin{tabular}{ccccc}
    \toprule
     Problem Set &  $|W|$ & $|F|$  & Capacity $u$ & Setup Cost $k$  \\
    \midrule
    IV & 16 & 50 & 5000 & 7.5 / 12.5 / 17.5 / 25\\
     VI & 16 & 50 & 15000 &  12.5 \\
    \bottomrule
    \end{tabular}%
  \label{table:fl-ins}%
\end{table}%

\subsubsection{Multi-commodity Network Design Problem (CMND).}

Consider a directed network with node set $N$, arc set $A$, and commodity set $K$. An uncertain $\tilde{v}_k$ amount of commodity $k, \forall k \in K$ must be routed from an origin node, $o_k \in N$, to a destination node, $d_k \in N$. 
The installation cost and arc capacity of arc $(i,j),\ \forall (i,j) \in A$ are $f_{ij}$ and $u_{ij}$, respectively.
The cost for transporting one unit of commodity $k, \ \forall k \in K$ on installed arc $(i,j),  \ \forall (i,j) \in A$ is $c_{ij}^k$. 
One needs to decide a subset of arcs to install before the realization of the demand to minimize the expected total cost.
We provide details about the RMP and SP formulations in Appendix~\ref{app:cmnd}.

We use the problem sets in \citet{crainic2014partial}, i.e., five problem sets (IV--VIII) from the set of R instances in \citet{crainic2011progressive}. Each problem set uses the same network, with parameters of each network shown in Table~\ref{table:mc-ins}.
The instances were originally proposed for the deterministic fixed charge multi-commodity network design problem \citep{crainic2001bundle}.
We apply techniques in \citet{song2014chance} to generate random samples. 
The demand $\tilde{v}_k$ of commodity $k$ in each scenario $\omega$ follows a Normal distribution with mean equal to the demand in the deterministic instances and standard deviation equal to $0.1$--$0.2$ times the mean.
\begin{table}[htbp]
  \centering
  \caption{Instance Attributes of CMND}
    \begin{tabular}{cccc}
    \toprule
     Problem Set &  $|N|$ & $|A|$  & $|K|$  \\
    \midrule
    IV & 10 & 60 & 10\\
    V & 10 & 60 & 25\\
    VI & 10 & 60 & 50\\
    VII & 10 & 82 & 10\\
    VIII & 10 & 83 & 25\\
    \bottomrule
    \end{tabular}%
  \label{table:mc-ins}%
\end{table}%

\subsection{Performance of the SVM Classifier}\label{subsec:results svm}

We first take Instance cap41 of CFLP and Instance r082 of CMND as examples to demonstrate the performance of the SVM classifier.
We present the prediction accuracy for cuts in each data set in Table~\ref{tab:svm accuracy}.
For each instance, we present the accuracy of five sets.
The first set is the training data set, which contains the cuts sampled in Phase 1.
We create four validation data sets consisting of unseen cuts.
To compute the prediction accuracy, we need to compute the ``true labels'' of cuts in the validation set by the label transformation function~\eqref{eq:label}. Therefore, all the cuts in validation sets are collected in the same way as that of Phase 1 (see Algorithm~\ref{alg:learn1}), but using different parameters (as shown in Table~\ref{tab:svm accuracy}).
The validation data set $1$ shares the same parameter as those of the training data set.
The validation data set $2$ uses the twice standard deviation as that of the training data set to generate realizations of the uncertain parameters for the optimization model in different scenarios.
The validation data set $3$ doubles the number of sampling paths, which can be equivalently viewed as a set sharing the same property with validation data set $1$ but with a twice larger size.
The validation data set $4$ uses a larger length of the sampling path, which can be viewed as a set containing more types of cuts, i.e., the validation data set $4$ also includes cuts generated in later iterations in addition to the cuts in the training data set.
\begin{table}[htbp]
  \centering
  \caption{Prediction Accuracy of the SVM Classifier}
     \resizebox{\textwidth}{!}{
    \begin{tabular}{lcccccccccc}
    \toprule
    \multicolumn{1}{c}{\multirow{4}[4]{*}{Inst.}} & \multicolumn{10}{c}{Prediction Accuracy ($ \%$)} \\
\cmidrule{2-11}          & \multicolumn{2}{c}{Training Data} & \multicolumn{2}{c}{Validation Data 1} & \multicolumn{2}{c}{Validation Data 2} & \multicolumn{2}{c}{Validation Data 3} & \multicolumn{2}{c}{Validation Data 4} \\
          & \multicolumn{2}{c}{Std $= 0.1$} & \multicolumn{2}{c}{Std  $= 0.1$} & \multicolumn{2}{c}{Std  $= 0.2$} & \multicolumn{2}{c}{Std  $= 0.1$} & \multicolumn{2}{c}{Std  $= 0.1$} \\
          & \multicolumn{1}{c}{$K=2$} & \multicolumn{1}{c}{$K=2$} & \multicolumn{1}{c}{$K=2$} & \multicolumn{1}{c}{$K=2$} & \multicolumn{1}{c}{$K=2$} & \multicolumn{1}{c}{$K=2$} & \multicolumn{1}{c}{$K=4$} & \multicolumn{1}{c}{$N=200$} & \multicolumn{1}{c}{$K=2$} & \multicolumn{1}{c}{$N=300$} \\
    \midrule
    cap41 & \multicolumn{2}{c}{$99.78$} & \multicolumn{2}{c}{78.25} & \multicolumn{2}{c}{67.50} & \multicolumn{2}{c}{75.00} & \multicolumn{2}{c}{78.33} \\
    r082 & \multicolumn{2}{c}{98.15} & \multicolumn{2}{c}{82.95} & \multicolumn{2}{c}{74.90} & \multicolumn{2}{c}{83.57} & \multicolumn{2}{c}{82.10} \\
    \bottomrule
    \end{tabular}%
    }
  \label{tab:svm accuracy}%
\end{table}%

The results in Table \ref{tab:svm accuracy} show that the in-sample prediction accuracy is higher than $98\%$ while the out-of-sample prediction accuracy of the validation data sets is relatively lower.
The out-of-sample prediction accuracy of validation data sets $1$, $2$, and $4$ is at similar levels higher than $75\%$ for both instances, and the prediction accuracy of Instance r082 is higher.
The out-of-sample prediction accuracy of validation data set $2$ is the lowest among the validation sets.
This is because the uncertain model parameters across scenarios of the validation data set $2$ are different from that of the training data set, and thus the generalization power is weaker than other validation sets.
Please note that the ``true labels'' of the validation data sets are computed by \eqref{eq:label}, which are our belief but not the exact classification of valuable or non-valuable cuts.

\subsection{Results of Comparing LearnBD with Benders Decomposition}\label{subsec:pcstat} 

As mentioned in Section~\ref{subsec:research}, the main obstacle of the traditional Benders is the large size of RMPs and, consequently, the long CPU time spent on solving RMPs in each iteration. 
SPs have smaller sizes and linear programming structures, and they can be efficiently solved in parallel.
Thus, the total time for solving a testing problem with Bender's approach is almost the cumulative time for solving RMPs. 
Therefore, we refer to algorithm efficiency as the RMP solving time in the rest of the paper.
In this section, we present the computational results for solving the CFLP and CMND instances with traditional Benders decomposition (BD) and LearnBD. Specifically, we show the number of iterations, optimality gap, number of cuts, and cumulative time of solving RMPs (i.e., the last four columns in Table~\ref{tab:fl-table},  Table~\ref{tab:mc-table}, and Table~\ref{tab:mc-table-precision}).
For LearnBD, we create a training problem for collecting cut information and we re-use this information when solving the testing problems. 
For all instances, we set the initial value of $\Delta$ as $1.2$ and decrease it by $0.01$ for each retraining conducted, i.e., $\mathcal{L}_\Delta= [1.2 - 0.01i]$ for $i=0, \ldots, 50$. 

\subsubsection{Results of CFLP.}\label{subsec:results cflp}

We compare the results of solving CFLP instances using LearnBD and BD in Table~\ref{tab:fl-table}.
Based on the problem size and the computational difficulty, we set the precision parameter $\delta= 0.01\%$ and the time limit as one hour.
The solving process terminates when $\frac{\textrm{UB-LB}}{\textrm{LB}}< \delta$ (see Step~\ref{step:while} in Algorithm~\ref{alg:learn3}) or the cumulative solving time of the RMP reaches the time limit.
\begin{table}[htbp]
  \centering
  \caption{Results of CFLP Instances Solved by LearnBD and BD}
   \resizebox{\textwidth}{!}{
    \begin{tabular}{cccrccrrrr}
    \toprule
    \multicolumn{1}{p{3.855em}}{Problem} & Instance & \multicolumn{1}{p{5.5em}}{Std\_Training} & $|\Omega|$  & \multicolumn{1}{p{5.215em}}{Std\_Testing} & Method &  Number   & {Opt Gap} & Number  & \multicolumn{1}{p{5.215em}}{Total Time} \\
    Set & & ($\times$ mean) & & ($\times$ mean) & & of Iter. & ($\%$) & of Cuts & of RMPs (s)\\
    \midrule
    \multirow{16}[16]{*}{IV} & \multirow{4}[4]{*}{cap41} & \multirow{4}[4]{*}{0.1} & \multirow{4}[4]{*}{100} & \multirow{2}[2]{*}{0.1} & BD    & 40    &  --  & 4000  & 183.48 \\
          &       &       &       &       & LearnBD & 38    & --  & 3790  & 96.31 \\
\cmidrule{5-10}          &       &       &       & \multirow{2}[2]{*}{0.2} & BD    & 47    & --  & 4700  & 187.05 \\
          &       &       &       &       & LearnBD & 45    & --  & 4467  & 141.30 \\
\cmidrule{2-10}          & \multirow{4}[4]{*}{cap42} & \multirow{4}[4]{*}{0.1} & \multirow{4}[4]{*}{200} & \multirow{2}[2]{*}{0.1} & BD    & 30   & --  & 6000  & 111.96 \\
          &       &       &       &       & LearnBD & 32   & --  & 4304  & 84.28 \\
\cmidrule{5-10}          &       &       &       & \multirow{2}[2]{*}{0.2} & BD    & 30    & --  & 6000  & 93.93 \\
          &       &       &       &       & LearnBD & 29    & --  & 5796  & 86.72 \\
\cmidrule{2-10}          & \multirow{4}[4]{*}{cap43} & \multirow{4}[4]{*}{0.1} & \multirow{4}[4]{*}{400} & \multirow{2}[2]{*}{0.1} & BD    & 27    & --  & 10800  & 243.57 \\
          &       &       &       &       & LearnBD &    26 & --  & 10378  & 220.18 \\
\cmidrule{5-10}          &       &       &       & \multirow{2}[2]{*}{0.2} & BD    & 26    & --  & 10400  & 228.66 \\
          &       &       &       &       & LearnBD & 26    & --  & 10369  & 215.07 \\
\cmidrule{2-10}          & \multirow{4}[4]{*}{cap44} & \multirow{4}[4]{*}{0.1} & \multirow{4}[4]{*}{400} & \multirow{2}[2]{*}{0.1} & BD    & 24    & --  & 9600  & 235.34 \\
          &       &       &       &       & LearnBD & 25    & --  & 9938  & 156.52 \\
\cmidrule{5-10}          &       &       &       & \multirow{2}[2]{*}{0.2} & BD    & 22    & --  & 8800  & 172.99 \\
          &       &       &       &       & LearnBD & 22    & --  & 8764  & 134.11 \\
    \midrule
    \multirow{4}[4]{*}{VI} & \multirow{4}[4]{*}{cap62} & \multirow{4}[4]{*}{0.1} & \multirow{4}[4]{*}{50} & \multirow{2}[2]{*}{0.1} & BD    & 258   & 3.26  & 12900 & LIMIT \\
          &       &       &       &       & LearnBD & 215   & 2.54  & 10701  & LIMIT \\
\cmidrule{5-10}          &       &       &       & \multirow{2}[2]{*}{0.2} & BD    & 206   & 3.06  & 10300  & LIMIT \\
          &       &       &       &       & LearnBD &  216  & 2.53  & 10786  & LIMIT \\
    \bottomrule
    \end{tabular}%
    }
  \label{tab:fl-table}%
\end{table}%

In Table~\ref{tab:fl-table}, (i) all instances in problem set IV can be optimized 
within the one-hour time limit. 
For these instances, the cuts generated by LearnBD are fewer than those of BD; the time of solving RMPs in LearnBD is shorter, and  the total time reduction ranges from $6\%$ (for cap43-0.2) to $47.5\%$ (for cap41-0.1). 
(ii) Instance cap62 is much more difficult to solve compared with other instances in problem set IV. Both BD and LearnBD exceed the one-hour time limit and LearnBD has a smaller optimality gap within one-hour time limit. The number of iterations and cuts of LearnBD are more than those of BD.
It happens that LearnBD executes more iterations because it adds fewer cuts than BD during each iteration.
For the second testing problem of the Instance cap62, LearnBD can solve RMPs with more cuts, which implies that the RMPs of LearnBD is easier to solve than the ones in BD.  In Table~\ref{tab:fl-table-prec}, we present the time and the number of cuts needed by LearnBD to reach the same or smaller optimality gap of BD shown in Table~\ref{tab:fl-table} for solving Instance cap62. Demonstrated in Table~\ref{tab:fl-table-prec}, LearnBD adds fewer cuts and takes much shorter time to achieve similar optimality gap as BD for the hard instance. 
(iii) For each instance, by comparing testing problems with two different standard deviations used for generating demand scenarios, we conclude that the improvement of LearnBD is more significant for testing problems having similar uncertainty distribution as the training problem.
One reason is that, for testing problems with larger standard deviation, the collected cuts in the training data in Phase 1 can be viewed as a subset of the total cut population that LearnBD encounters in Phase~2.
\begin{table}[htbp]
  \centering
  \caption{Comparing LearnBD with BD for Solving Instance cap62 to Similar Accuracy}
   \resizebox{\textwidth}{!}{
    \begin{tabular}{cccrccrrrr}
    \toprule
    \multicolumn{1}{p{3.855em}}{Problem} & Instance & \multicolumn{1}{p{5.5em}}{Std\_Training} & $|\Omega|$  & \multicolumn{1}{p{5.215em}}{Std\_Testing} & Method &  Number   & {Opt Gap} & Number  & \multicolumn{1}{p{5.215em}}{Total Time} \\
    Set & & ($\times$ mean) & & ($\times$ mean) & & of Iter. & ($\%$) & of Cuts& of RMPs (s)\\
    \midrule
    \multirow{4}[4]{*}{VI} & \multirow{4}[4]{*}{cap62} & \multirow{4}[4]{*}{0.1} & \multirow{4}[4]{*}{50} & \multirow{2}[2]{*}{0.1} & BD    & 258   & 3.26  & 12900 & $>$3600 \\
          &       &       &       &       & LearnBD & 168   & 3.12 & 8357  & 1851.10 \\
\cmidrule{5-10}          &       &       &       & \multirow{2}[2]{*}{0.2} & BD    & 206   & 3.06  & 10300  & $>$3600 \\
          &       &       &       &       & LearnBD &  194 & 3.05  & 9686  & 2674.99 \\
    \bottomrule
    \end{tabular}%
    }
  \label{tab:fl-table-prec}%
\end{table}%

\subsubsection{Results of CMND.}\label{subsec:results cmnd}
We present the results of LearnBD and BD for solving CMND instances in Table~\ref{tab:mc-table},
where the optimality tolerance $\delta=1\%$ and the time limit is set as two hours.
\begin{table}[htbp]
  \centering
  \caption{Results of CMND Instances Solved by LearnBD and BD}
   \resizebox{\textwidth}{!}{
    \begin{tabular}{cccrccrrrr}
    \toprule
    \multicolumn{1}{p{3.855em}}{Problem} & Instance & \multicolumn{1}{p{5.5em}}{Std\_Training} & $|\Omega|$  & \multicolumn{1}{p{5.215em}}{Std\_Testing} & Method &  Number   & {Opt Gap} & Number  & \multicolumn{1}{p{5.215em}}{Total Time} \\
    Set & & ($\times$ mean) & & ($\times$ mean) & & of Iter. & ($\%$)& of Cuts& of RMPs (s)\\
    \midrule
 \cmidrule{1-5}    \multirow{8}[8]{*}{IV} & \multirow{4}[4]{*}{r041} & \multirow{4}[4]{*}{0.1} & \multirow{4}[4]{*}{80} & \multirow{2}[2]{*}{0.1} & BD    & 269    & 42.19 & 21520  & LIMIT \\
          &       &       &       &       & LearnBD & 248    & 22.88 & 19761  & LIMIT \\
\cmidrule{5-10}          &       &       &       & \multirow{2}[2]{*}{0.2} & BD    & 275    & 20.44 & 22000  & LIMIT \\
          &       &       &       &       & LearnBD & 244   & 18.18 & 19283  & LIMIT \\
\cmidrule{2-10}          & \multirow{4}[4]{*}{r046} & \multirow{4}[4]{*}{0.1} & \multirow{4}[4]{*}{80} & \multirow{2}[2]{*}{0.1} & BD    & 32    & --  & 2560  & 62.30 \\
          &       &       &       &       & LearnBD & 28    & --  & 1708   & 41.70 \\
\cmidrule{5-10}          &       &       &       & \multirow{2}[2]{*}{0.2} & BD    & 28    & --  & 2240  & 111.53 \\
          &       &       &       &       & LearnBD &35    & --  & 2094   & 80.81 \\
    \midrule
    \multirow{8}[8]{*}{V} & \multirow{4}[4]{*}{r051} & \multirow{4}[4]{*}{0.1} & \multirow{4}[4]{*}{100} & \multirow{2}[2]{*}{0.1} & BD    & 185   & 9.79 & 18500 & LIMIT \\
          &       &       &       &       & LearnBD & 155    & 3.66  & 15210  & LIMIT \\
\cmidrule{5-10}          &       &       &       & \multirow{2}[2]{*}{0.2} & BD    & 193   & 10.37 & 19300 & LIMIT \\
          &       &       &       &       & LearnBD & 190   & 8.69  & 18803 & LIMIT \\
\cmidrule{2-10}          & \multirow{4}[4]{*}{r054} & \multirow{4}[4]{*}{0.1} & \multirow{4}[4]{*}{100} & \multirow{2}[2]{*}{0.1} & BD    & 64    & --  & 6400  & 585.944 \\
          &       &       &       &       & LearnBD & 56    & --  & 5398  & 385.89 \\
\cmidrule{5-10}          &       &       &       & \multirow{2}[2]{*}{0.2} & BD    & 85    & -- & 8500  & 1172.40 \\
          &       &       &       &       & LearnBD & 78    & --  & 7041  & 984.71 \\
    \midrule
    \multirow{4}[4]{*}{VI} & \multirow{4}[4]{*}{r061} & \multirow{4}[4]{*}{0.1} & \multirow{4}[4]{*}{100} & \multirow{2}[2]{*}{0.1} & BD    & 121   & 10.54  & 12100 & LIMIT \\
          &       &       &       &       & LearnBD & 125    & 9.01  & 12288  & LIMIT \\
\cmidrule{5-10}          &       &       &       & \multirow{2}[2]{*}{0.2} & BD    & 137   & 10.28  & 13700 & LIMIT \\
          &       &       &       &       & LearnBD & 135    & 9.69 & 13602  & LIMIT \\
    \midrule
    \multirow{12}[12]{*}{VII} & \multirow{4}[4]{*}{r071} & \multirow{4}[4]{*}{0.1} & \multirow{4}[4]{*}{80} & \multirow{2}[2]{*}{0.1} & BD    & 159   & 958.69 & 12720 & LIMIT \\
          &       &       &       &       & LearnBD & 170   & 884.19 & 13194 & LIMIT \\
\cmidrule{5-10}          &       &       &       & \multirow{2}[2]{*}{0.2} & BD    & 154   & 967.37 & 12320 & LIMIT \\
          &       &       &       &       & LearnBD & 167   & 885.94 & 13215 & LIMIT \\
\cmidrule{2-10}          & \multirow{4}[4]{*}{r075} & \multirow{4}[4]{*}{0.1} & \multirow{4}[4]{*}{80} & \multirow{2}[2]{*}{0.1} & BD    & 121   & 10.89  & 9680 & LIMIT \\
          &       &       &       &       & LearnBD & 150   & 9.24  & 9621  & LIMIT \\
\cmidrule{5-10}          &       &       &       & \multirow{2}[2]{*}{0.2} & BD    & 119   & 11.23  & 9520 & LIMIT \\
          &       &       &       &       & LearnBD & 126   & 10.07 & 9350  & LIMIT \\
\cmidrule{2-10}          & \multirow{4}[4]{*}{r076} & \multirow{4}[4]{*}{0.1} & \multirow{4}[4]{*}{80} & \multirow{2}[2]{*}{0.1} & BD    & 38    & --  & 3040  & 360.93 \\
          &       &       &       &       & LearnBD & 52    & --  & 3026  & 276.91 \\
\cmidrule{5-10}          &       &       &       & \multirow{2}[2]{*}{0.2} & BD    & 38    & --  & 3040  & 374.81 \\
          &       &       &       &       & LearnBD & 48    & --  & 3080  & 305.71 \\
    \midrule
    \multirow{4}[4]{*}{VIII} & \multirow{4}[4]{*}{r082} & \multirow{4}[4]{*}{0.1} & \multirow{4}[4]{*}{80} & \multirow{2}[2]{*}{0.1} & BD    & 106    & 9.77  & 8480  & LIMIT \\
          &       &       &       &       & LearnBD & 113    & 9.27  &   8873  & LIMIT \\
\cmidrule{5-10}          &       &       &       & \multirow{2}[2]{*}{0.2} & BD    &104    & 9.91  & 8320 & LIMIT \\
          &       &       &       &       & LearnBD & 106    & 7.13  & 8156  & LIMIT \\
    
    \bottomrule
    \end{tabular}%
    }
  \label{tab:mc-table}%
\end{table}%

In Table~\ref{tab:mc-table}, most CMND instances take longer time than CFLP instances. We have similar observations as in Table \ref{tab:fl-table}: 
(i) Instances r046, r054, and r076 can be optimized within the two-hour time limit and the cumulative solving time of RMPs of LearnBD is significantly less than that of BD. The cuts in LearnBD are fewer than cuts generated by BD for Instances r046 and r076.
(ii) The rest of instances are hard to solve and both LearnBD and BD reach the two-hour time limit before converging. Instance r071 is extremely hard to solve and the optimality gaps of the two testing problems are greater than $100\%$ when reaching time limit.
For all of these instances, LearnBD obtains a smaller optimality gap. In Table~\ref{tab:mc-table-precision}, we further show the computational time and number of cuts needed by LearnBD to achieve similar optimality gap as BD for the instances that were not solved to optimality either by BD or LearnBD in Table~\ref{tab:mc-table}. 

\begin{table}[htbp]
  \centering
  \caption{Comparing LearnBD with BD for Solving CMND instances to Similar Accuracy}
   \resizebox{\textwidth}{!}{
    \begin{tabular}{cccrccrrrr}
    \toprule
    \multicolumn{1}{p{3.855em}}{Problem} & Instance & \multicolumn{1}{p{5.5em}}{Std\_Training} & $|\Omega|$  & \multicolumn{1}{p{5.215em}}{Std\_Testing} & Method &  Number   & {Opt Gap} & Number  & \multicolumn{1}{p{5.215em}}{Total Time} \\
    Set & & ($\times$ mean) & & ($\times$ mean) & & of Iter. & ($\%$)& of Cuts& of RMPs (s)\\
    \midrule
 \cmidrule{1-5}    \multirow{4}[4]{*}{IV} & \multirow{4}[4]{*}{r041} & \multirow{4}[4]{*}{0.1} & \multirow{4}[4]{*}{80} & \multirow{2}[2]{*}{0.1} & BD    & 269    & 42.19 & 21520  & $>$ 7200 \\
          &       &       &       &       & LearnBD & 204    & 27.41 & 13241  & 4140.49 \\
\cmidrule{5-10}          &       &       &       & \multirow{2}[2]{*}{0.2} & BD    & 275    & 20.44 & 22000  & $>$ 7200 \\
          &       &       &       &       & LearnBD & 158   & 20.39 & 12403 & 2505.91 \\
    \midrule
    \multirow{4}[4]{*}{V} & \multirow{4}[4]{*}{r051} & \multirow{4}[4]{*}{0.1} & \multirow{4}[4]{*}{100} & \multirow{2}[2]{*}{0.1} & BD    & 185   & 9.79 & 18500 & $>$ 7200 \\
          &       &       &       &       & LearnBD & 63    & 7.67  & 6010  & 249.52 \\
\cmidrule{5-10}          &       &       &       & \multirow{2}[2]{*}{0.2} & BD   & 193   & 10.37 & 19300 & $>$ 7200 \\
          &       &       &       &       & LearnBD & 119  & 10.36  & 11764 & 2197.31 \\
\midrule
    \multirow{4}[4]{*}{VI} & \multirow{4}[4]{*}{r061} & \multirow{4}[4]{*}{0.1} & \multirow{4}[4]{*}{100} & \multirow{2}[2]{*}{0.1} & BD    & 121   & 10.54  & 12100 & $>$ 7200 \\
          &       &       &       &       & LearnBD & 71    & 10.45  & 6902  & 1129.37 \\
\cmidrule{5-10}          &       &       &       & \multirow{2}[2]{*}{0.2} & BD    & 137   & 10.28  & 13700 & $>$ 7200 \\
          &       &       &       &       & LearnBD & 70    & 10.26 & 6835  & 1030.30 \\
    \midrule
    \multirow{8}[8]{*}{VII} & \multirow{4}[4]{*}{r071} & \multirow{4}[4]{*}{0.1} & \multirow{4}[4]{*}{80} & \multirow{2}[2]{*}{0.1} & BD    & 159   & 958.69 & 12720 & $>$ 7200 \\
          &       &       &       &       & LearnBD & 143   & 937.63 & 11034 & 4868.40 \\
\cmidrule{5-10}          &       &       &       & \multirow{2}[2]{*}{0.2} & BD    & 154   & 967.37 & 12320 & $>$ 7200 \\
          &       &       &       &       & LearnBD & 139   & 966.31 & 10975 & 4816.33 \\
\cmidrule{2-10}          & \multirow{4}[4]{*}{r075} & \multirow{4}[4]{*}{0.1} & \multirow{4}[4]{*}{80} & \multirow{2}[2]{*}{0.1} & BD    & 121   & 10.89  & 9680 & $>$ 7200 \\
          &       &       &       &       & LearnBD & 115   & 10.86  & 6821 & 2416.36 \\
\cmidrule{5-10}          &       &       &       & \multirow{2}[2]{*}{0.2} & BD    & 119   & 11.23  & 9520 & $>$ 7200 \\
          &       &       &       &       & LearnBD & 103   & 11.19 & 7510  & 3641.27 \\
    \midrule
    \multirow{4}[4]{*}{VIII} & \multirow{4}[4]{*}{r082} & \multirow{4}[4]{*}{0.1} & \multirow{4}[4]{*}{80} & \multirow{2}[2]{*}{0.1} & BD    & 106    & 9.77  & 8480  & $>$ 7200 \\
          &       &       &       &       & LearnBD & 55 & 9.75 & 4233  & 590.09  \\
\cmidrule{5-10}          &       &       &       & \multirow{2}[2]{*}{0.2} & BD    &104    & 9.91  & 8320 & $>$ 7200 \\
          &       &       &       &       & LearnBD & 54   & 8.72  & 4158 & 658.68 \\
    
    \bottomrule
    \end{tabular}%
    }
  \label{tab:mc-table-precision}%
\end{table}%

In Table~\ref{tab:mc-table-precision}, LearnBD uses much shorter time to obtain a similar (slightly tighter) gap as the gap that BD achieves under the two-hour time limit for all instances.
And similarly, LearnBD performs better in testing problems which have similar uncertainty distribution as the training problems.
For instances r051 and r082, LearnBD achieves a similar optimality gap within $10\%$ of the time spent by BD. 
Comparing the optimality gap results and the solving time of LearnBD versus BD in Tables~\ref{tab:mc-table} and~\ref{tab:mc-table-precision}, we notice that BD takes longer to attain solutions with similar optimality gaps as the ones of LearnBD for the very hard instances that cannot be optimized within the 2-hour time limit by either method in Table~\ref{tab:mc-table}. 
Therefore, we conclude that for instances which are hard to solve: (i) when we allow a relatively large optimality gap, for example, $10\%$ for Instance r051, 
LearnBD can achieve this gap in significantly shorter time compared to BD; and (ii) when we allow a relatively small gap, both LearnBD and BD need several iterations and thus long solving time. Note that LearnBD only takes $249$ seconds while BD takes more than $7200$ seconds to solve Instance r051 to a less than $10\%$ gap.

\begin{remark}\label{remark:compu retraining times}
Based on our computational results, retraining almost happens in every instance and $\Delta$ may decrease consecutively over iterations before the classifier can identify valuable cuts, i.e., we may consecutively perform retraining.
As mentioned in Remark \ref{remark:train time}, the retraining can be done before the solving process and thus does not affect the total solving time.
The label predicting process of cuts is extremely fast and can be implemented in parallel.
Therefore, the time consumed by retraining, or consecutive retraining, is also negligible when considering the total solving time. 
In our computational tests, LearnBD decreases the values of $\Delta$ relatively more frequently in the first several iterations and then keeps using a fixed $\Delta$ until the termination.
For readers' interest, we present the values of $\Delta$ during iterations of five CMND instances in Table~\ref{tab:delta value}.

\begin{table}[htbp]
  \centering
  \caption{$\Delta$-values Throughout Computational Iterations of Five CMND Instances}
    \begin{tabular}{ccccccccccc}
    \toprule
    \multirow{2}[0]{*}{r041} & Iteration & 1--3   & 3--248 &       &       &       &       &       &       &  \\
    \cmidrule{2-4}
          & $\Delta$ & 1.20   & 1.13  &       &       &       &       &       &       &  \\
    \midrule
    \multirow{2}[0]{*}{r051} & Iteration & 1--2   & 3     & 4--155 &       &       &       &       &       &  \\
      \cmidrule{2-5}     & $\Delta$ & 1.07  & 1.06  & 1.01  &       &       &       &       &       &  \\
    \midrule
    \multirow{2}[0]{*}{r071} & Iteration & 1--5   & 6--170 &       &       &       &       &       &       &  \\
      \cmidrule{2-4}     & $\Delta$ & 1.20   & 1.12  &       &       &       &       &       &       &  \\
    \midrule
    \multirow{2}[0]{*}{r075} & Iteration & 1     & 2     & 3     & 4--5   & 6--7   & 8--15  & 16--150 &       &  \\
     \cmidrule{2-9}      & $\Delta$ & 1.20   & 1.19  & 1.18  & 1.13  & 1.07  & 1.06  & 1.00     &       &  \\
    \midrule
    \multirow{2}[1]{*}{r076} & Iteration & 1     & 2     & 3     & 4     & 5     & 6--7   & 8--10  & 11--15 & 16--52 \\
     \cmidrule{2-11}      & $\Delta$ & 1.20   & 1.11  & 1.06  & 1.05  & 1.02  & 1.01  & 1.00     & 0.99  & 0.98 \\
    \bottomrule
    \end{tabular}%
    
  \label{tab:delta value}%
\end{table}%
\end{remark}

\subsubsection{Results of Replicates for CFLP Instance.} 

In this section, we provide additional computational results for five replicates of each instance in the CFLP problem set IV to show the performance consistency of comparing LearnBD and BD across randomly sampled scenarios. 
The demand scenarios of five replicates of each instance are generated independently with the same standard deviation as the training problem. We present the minimum, maximum, mean, and median values of (i) the number of generated cuts and (i) the total solution time of RMPs in Table~\ref{tab:result-rep}.
The observations are consistent with those in Section~\ref{subsec:results cflp} and Section~\ref{subsec:results cmnd} that LearnBD can save more time than BD for optimizing the testing problems.

\begin{table}[htbp]
  \centering
  \caption{Results of Multiple Runs for CFLP Instance from Problem Set IV}
    \begin{tabular}{cccccccccc}
\toprule
\multirow{2}[4]{*}{Instance} & \multirow{2}[4]{*}{Method} & \multicolumn{4}{c}{Number of Cuts} & \multicolumn{4}{c}{Total Time of RMPs (s)} \\
\cmidrule{3-10}          &       & Min   & Max   & Mean  & Median & Min   & Max   & Mean  & Median \\
    \midrule
    \multirow{2}[2]{*}{cap41} & BD    & 3000  & 4000  & 3740  & 3900  & 124.38 & 209.68 & 161.89 & 154.11 \\
    \cmidrule{2-10} 
          & LearnBD & 2499  & 3790  & 3366.2 & 3446  & 81.08 & 173.07 & 106.55 & 96.31 \\
    \midrule
    \multirow{2}[2]{*}{cap42} & BD    & 5800  & 6000  & 5920  & 6000  & 71.70 & 124.95 & 99.95 & 104.79 \\
    \cmidrule{2-10} 
          & LearnBD & 4304  & 5798  & 5123.2 & 5177  & 63.10 & 100.56 & 78.61 & 81.39 \\
    \midrule
    \multirow{2}[2]{*}{cap43} & BD    & 10400 & 10800 & 10640 & 10800 & 147.02 & 263.42 & 211.64 & 221.44 \\
    \cmidrule{2-10} 
          & LearnBD & 10254 & 10486 & 10374 & 10378 & 145.02 & 246.86 & 200.44 & 214.07 \\
    \midrule
    \multirow{2}[2]{*}{cap44} & BD    & 8000  & 9600  & 8720  & 8800  & 171.10 & 249.28 & 213.06 & 217.24 \\
    \cmidrule{2-10} 
          & LearnBD & 8374  & 9938  & 8927.6 & 8774  & 135.97 & 174.92 & 159.06 & 156.52 \\
    \bottomrule
    \end{tabular}%
  \label{tab:result-rep}%
\end{table}%

\subsection{Performance Comparison Over Iterations}\label{subsec:pciteration}

To track the performance of BD and LearnBD over iterations to show their convergence, we depict and analyze the results of two specific instances in Section~\ref{subsubsec:cap41} and Section~\ref{subsubsec:r075}. For each instance, we solve it with both BD and LearnBD separately to the same optimality gap. 
For each approach, we record the following values after each iteration: 
\begin{itemize}
    \item the optimality gap after the current iteration, which helps track the algorithm convergence;
    \item the cumulative time for solving RMPs;
    \item the total number of cuts added to RMPs in the previous iterations, which reflects the size of RMPs and power of the classifier;
    \item the cumulative time for solving SPs, which can be performed in parallel in each iteration, and therefore this value will not affect the total solving time of the algorithm.
\end{itemize}

\subsubsection{CFLP: Problem  IV, cap41.}\label{subsubsec:cap41}

The results over the iterations are shown in Figure~\ref{fig:plot}.
The instance has 100 scenarios and the termination criterion is reaching $\delta=0.01\%$ optimality gap.
 BD requires $40$ iterations while LearnBD takes $38$ iterations.
LearnBD adds fewer cuts but achieves a similar optimality gap, which indicates the power of our SVM cuts classifier.
Due to the smaller sizes of RMPs, the cumulative time for solving RMPs in LearnBD is significantly shorter than that in BD.

\begin{figure}[htbp]
   \centering
    \begin{subfigure}[b]{0.48\textwidth}
        \includegraphics[width=\textwidth]{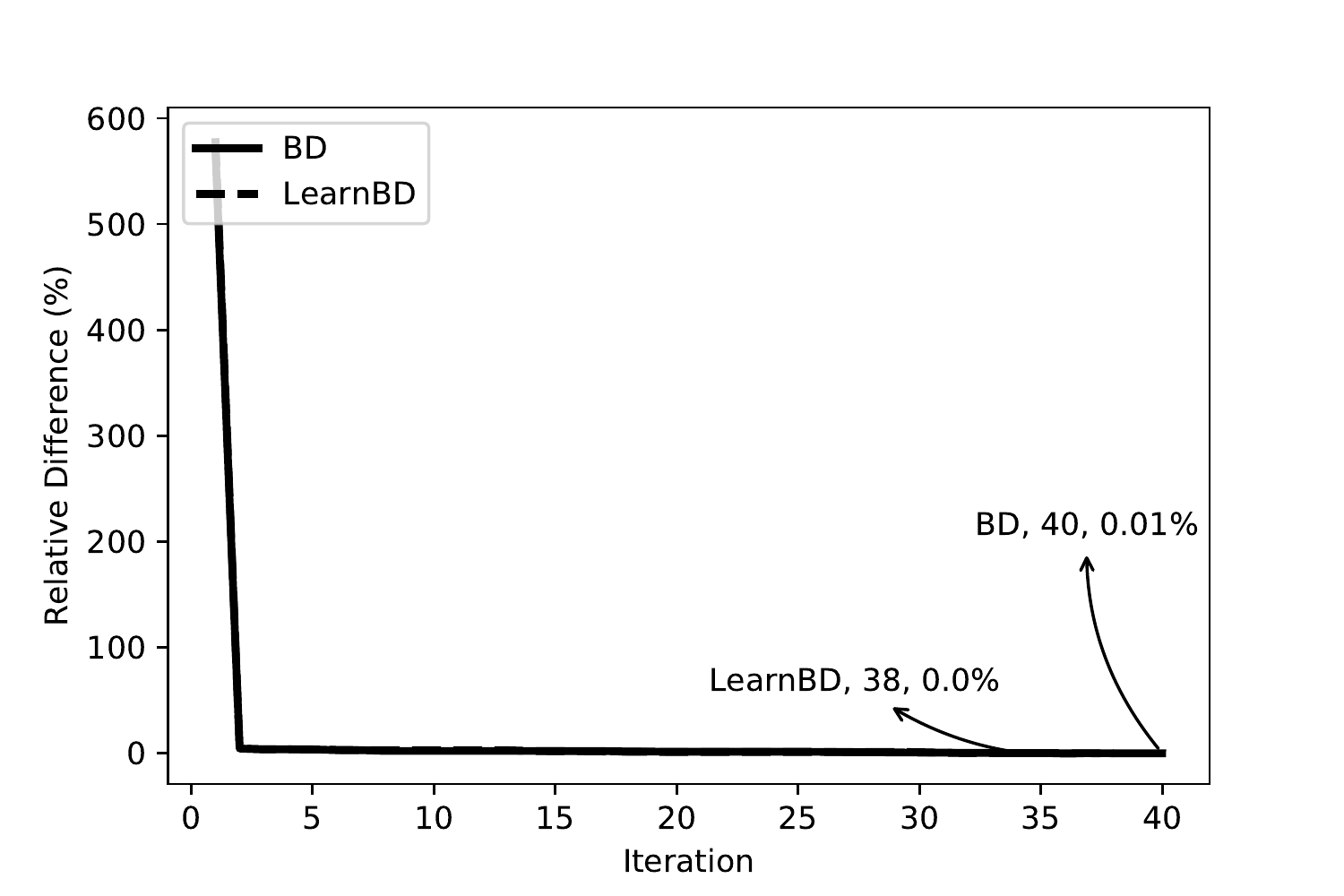}
        \caption{Optimality gap.}
        \label{fig:processgap}
    \end{subfigure}
    ~
    \begin{subfigure}[b]{0.48\textwidth}
        \includegraphics[width=\textwidth]{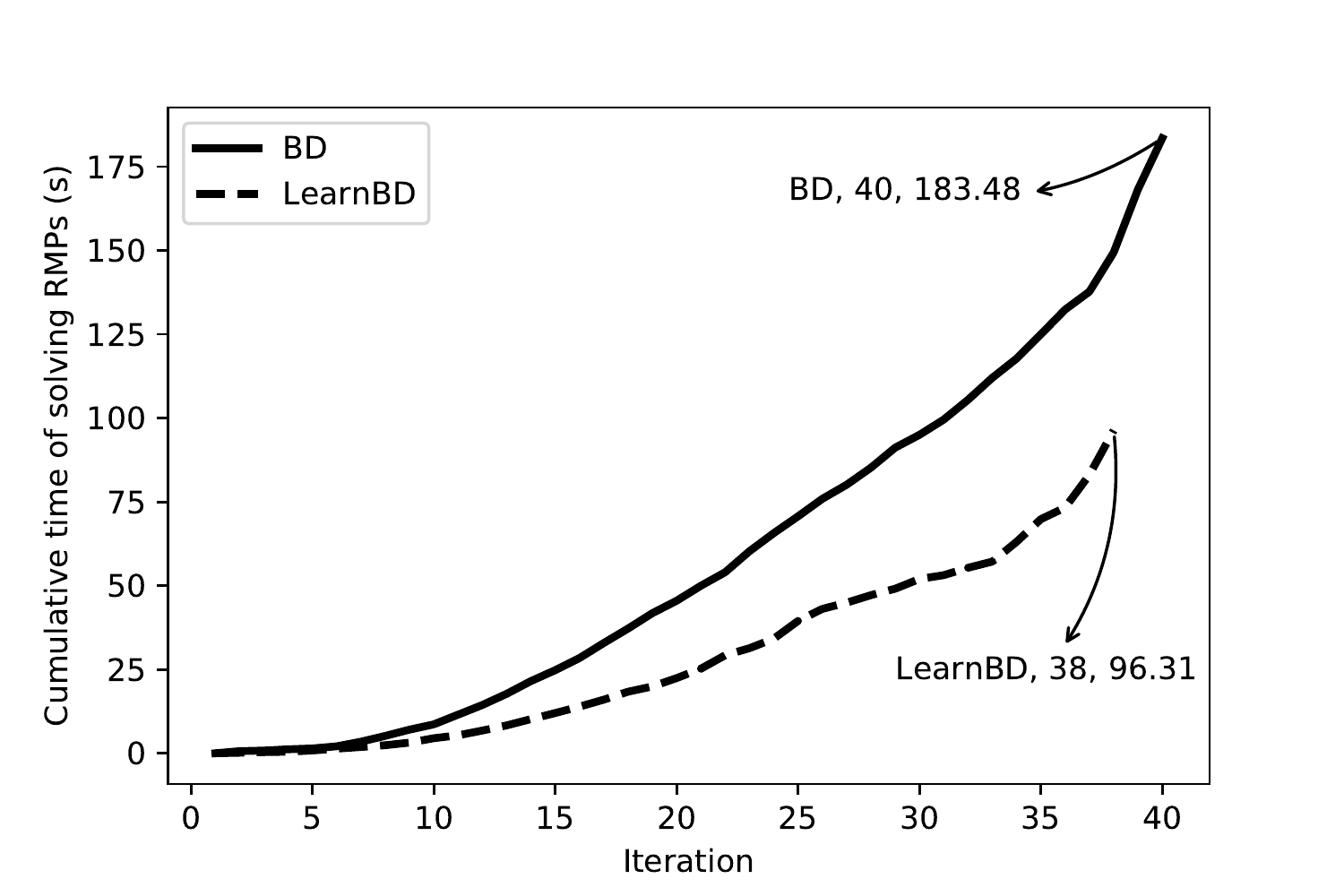}
        \caption{Cumulative time for solving RMPs.}
        \label{fig:processRMPt}
    \end{subfigure}
    
    \begin{subfigure}[b]{0.48\textwidth}
        \includegraphics[width=\textwidth]{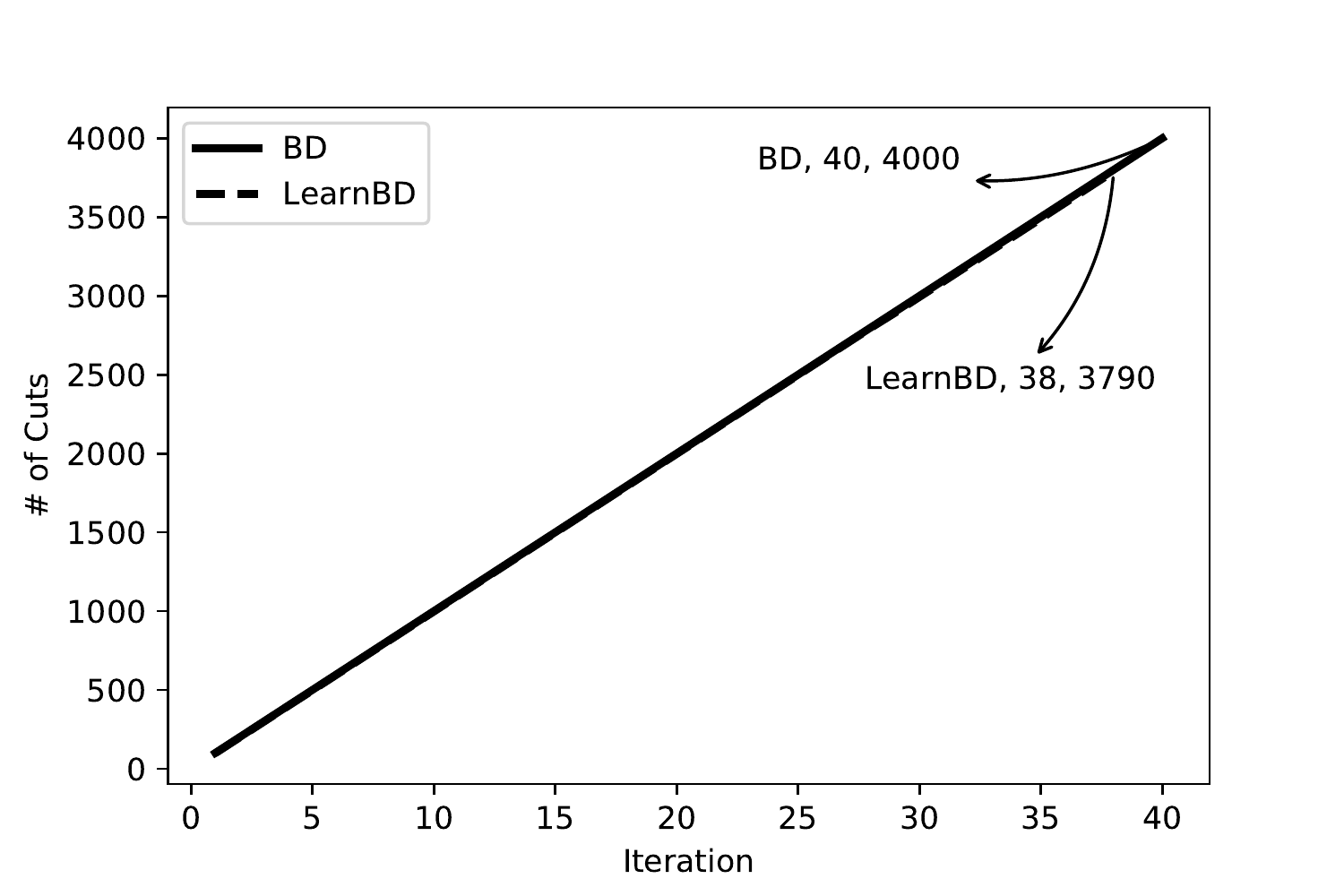}
        \caption{Number of cuts added to RMPs.}
        \label{fig:processncuts}
    \end{subfigure}
    ~
    \begin{subfigure}[b]{0.48\textwidth}
        \includegraphics[width=\textwidth]{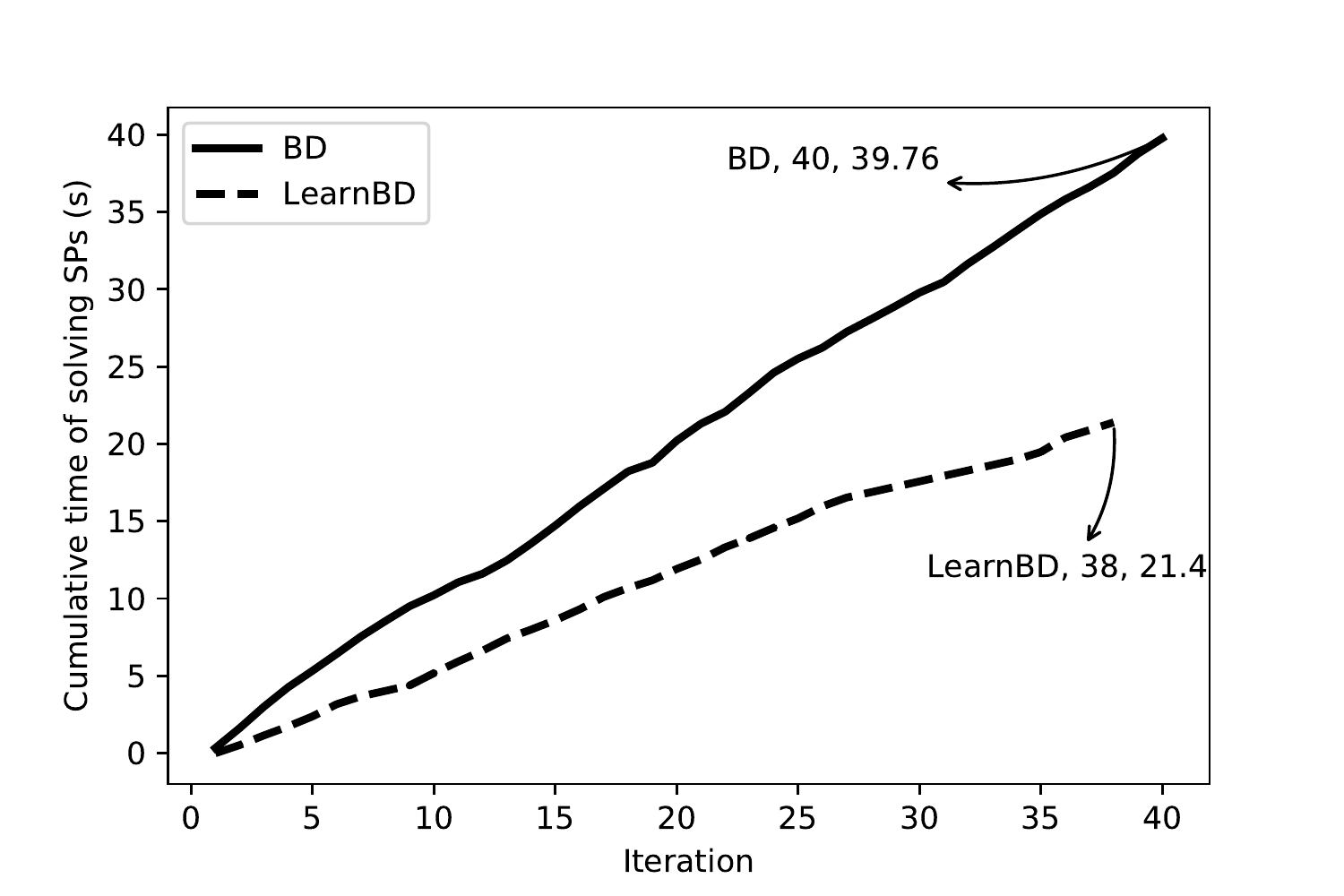}
        \caption{Cumulative time for solving SPs.}
        \label{fig:processSPt}
    \end{subfigure}

    \caption{CFLP: Problem set IV, cap41 instance solved by BD and LearnBD. The horizontal axis is the iteration number.}\label{fig:plot}
\end{figure}

\subsubsection{CMND: Problem Set VII, r075.}\label{subsubsec:r075}

The results over the iterations for both BD and LearnBD are shown in Figure~\ref{fig:plot2}. 
The instance has 80 scenarios and the termination criterion is reaching $\delta=10.89\%$ optimality gap (the result that BD achieves within the two-hour time limit).
We observe similar performance of LearnBD as in the case of solving the CFLP instance. 
In Figure~\ref{fig:processgap2}, in the first 40 iterations, the gap of LearnBD convergences slower than BD because it adds fewer cuts. After 100 iterations, both algorithms achieve similar gaps.
In Figure~\ref{fig:processncuts2}, in the later iterations, the number of added cuts per iteration is quite similar to that added by BD. 
Thus, we can conclude that the first few iterations are important for reducing the total solving time.
\begin{figure}[htbp]
   \centering
    \begin{subfigure}[b]{0.48\textwidth}
        \includegraphics[width=\textwidth]{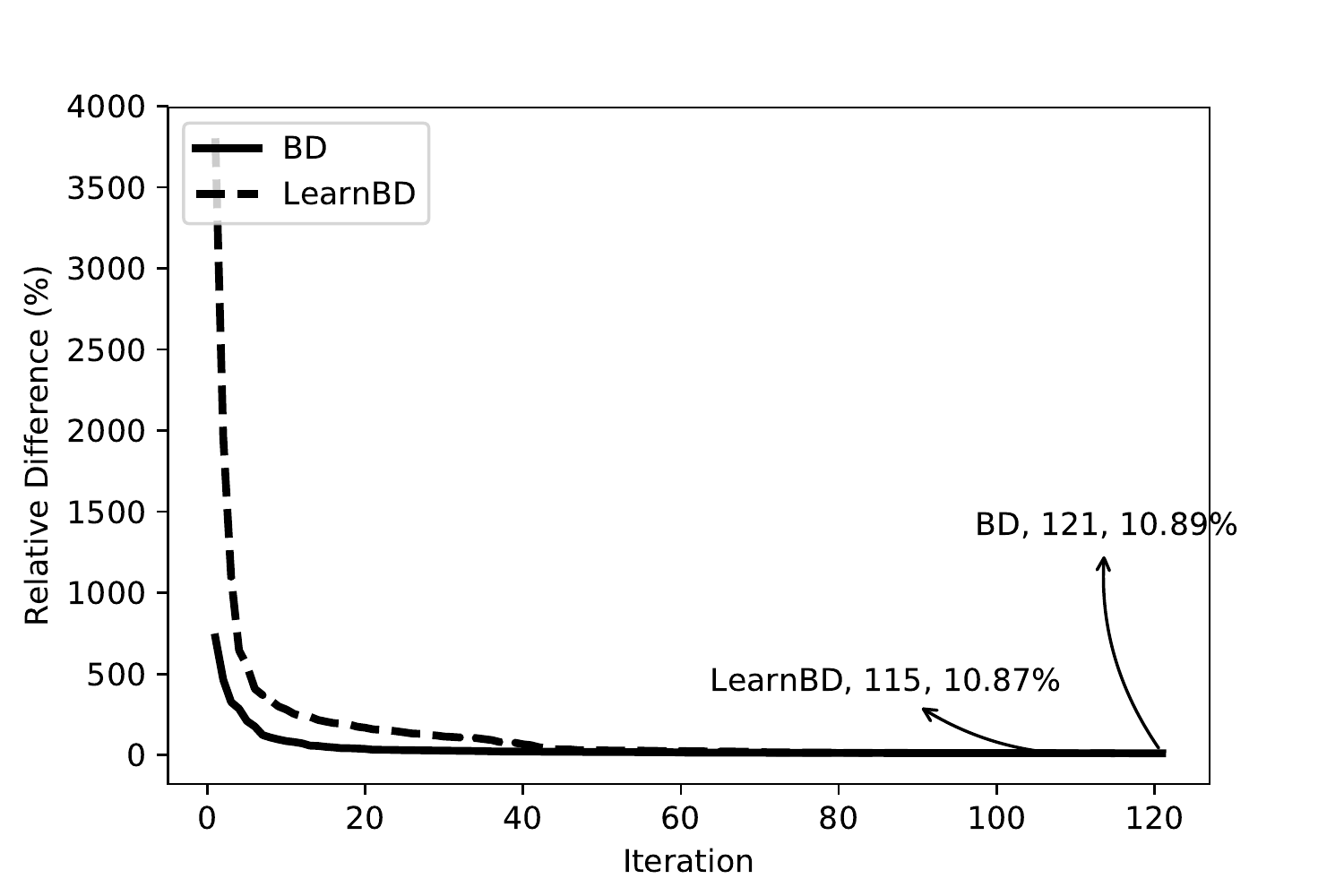}
        \caption{Optimality gap.}
        \label{fig:processgap2}
    \end{subfigure}
    ~
    \begin{subfigure}[b]{0.48\textwidth}
        \includegraphics[width=\textwidth]{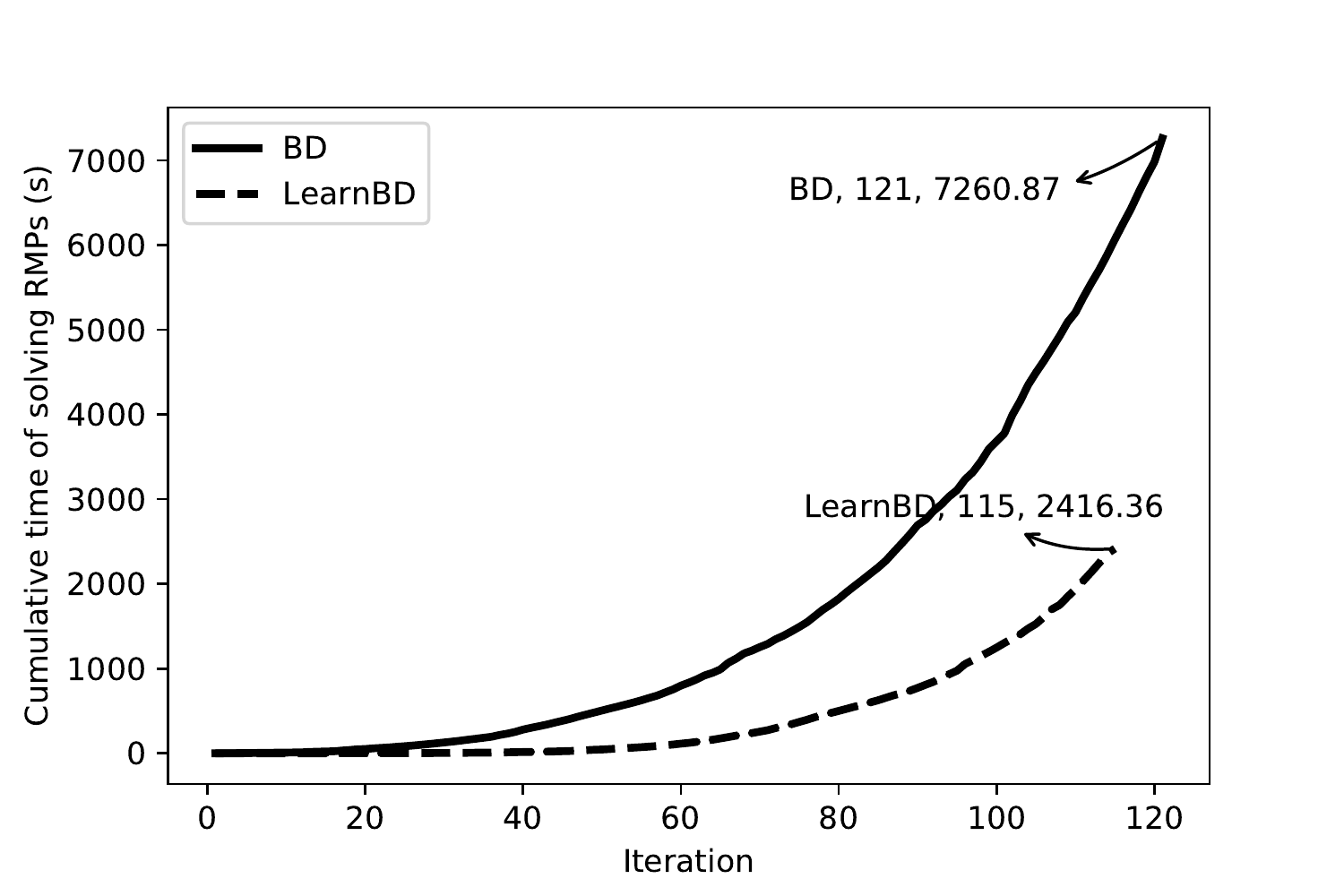}
        \caption{Cumulative time for solving RMPs.}
        \label{fig:processRMPt2}
    \end{subfigure}
    
    \begin{subfigure}[b]{0.48\textwidth}
        \includegraphics[width=\textwidth]{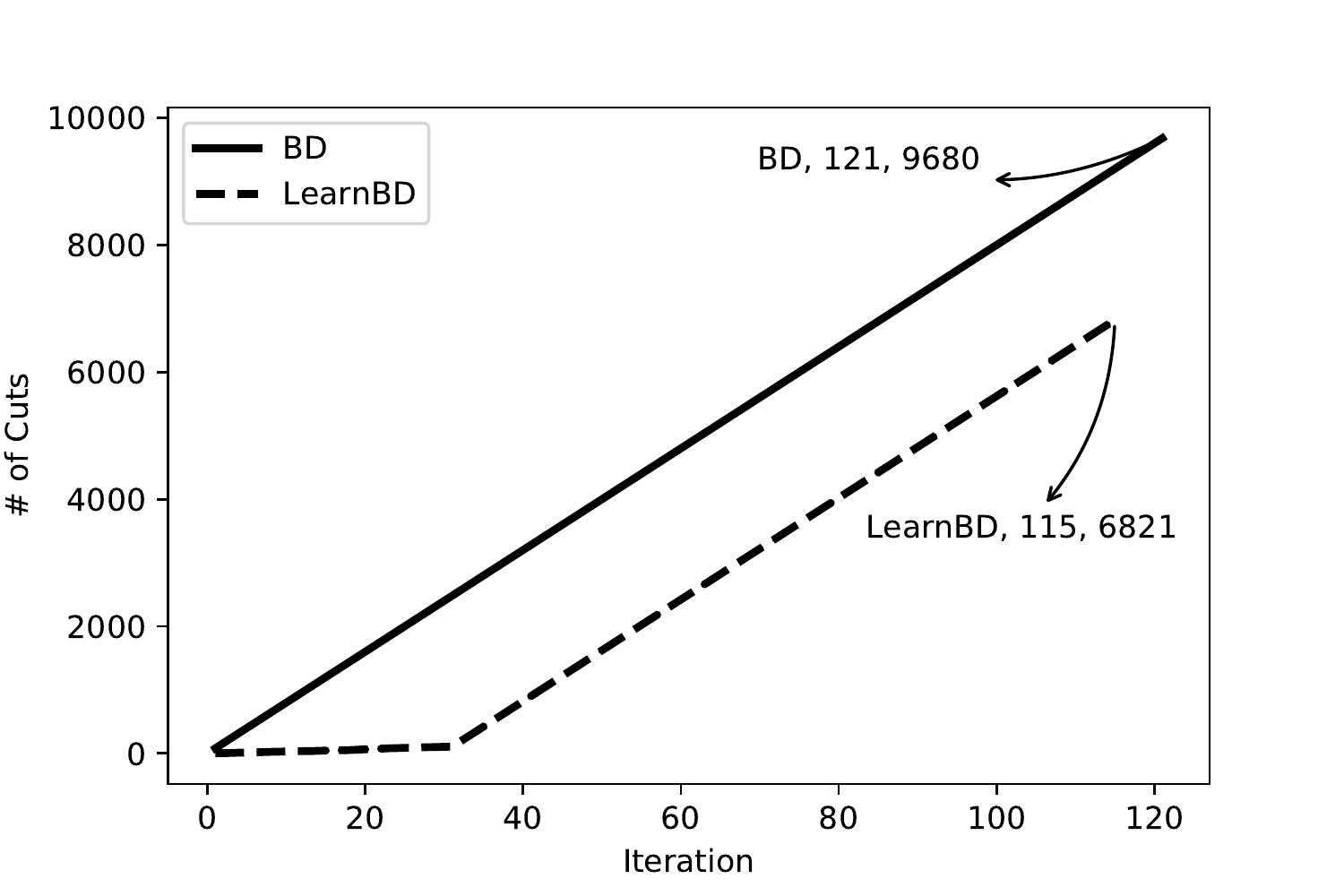}
        \caption{Number of cuts added to RMPs.}
        \label{fig:processncuts2}
    \end{subfigure}
    ~
    \begin{subfigure}[b]{0.48\textwidth}
        \includegraphics[width=\textwidth]{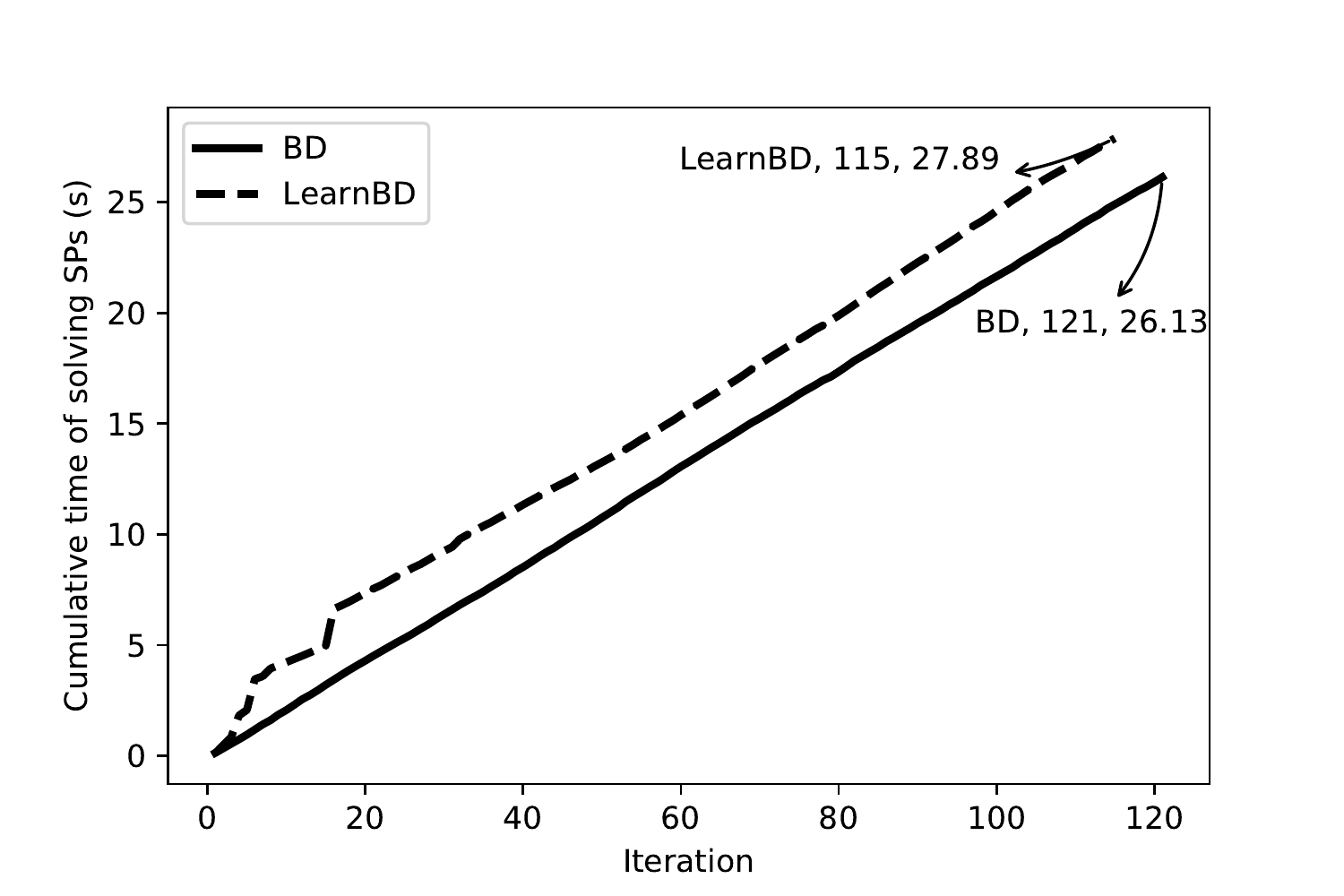}
        \caption{Cumulative time for solving SPs.}
        \label{fig:processSPt2}
    \end{subfigure}

    \caption{CMND: Problem set VII, r075 instance solved by BD and LearnBD. The horizontal axis is the iteration number.}\label{fig:plot2}
\end{figure}

\subsection{Sampling and Training Time}\label{subsec:training time}

As we mentioned in Remark~\ref{remark:phase 1 time} and Remark~\ref{remark:train time}, the cut sampling and classifier training processes can be done before starting solving the testing problems. The outcomes of these two processes can be repeatedly used for solving different testing problems. Therefore, the time of these two processes does not affect the total solving time of LearnBD. Here, we record and present the time of Phase 1 and the time of classifier training in Phase 2 in Table~\ref{tab:training time}.
Column $K$ shows the number of sampling paths and Column $N$ shows the length of the sampling paths, i.e., the number of iterations in each sampling path, in Phase 1. 
We include the following four columns showing the results of solving one testing problem of each instance: Method, Opt Gap, Number of Cuts, and Testing-Time. Note that the testing problems are different from the training problem and for each instance, we only include results of one testing problem with the same standard deviation as the training problem. The last two columns in Table~\ref{tab:training time} are: ``Phase-1-Time'' showing the total time of cut sampling and collecting information from training problems in Phase 1 and ``SVM-Training-Time'' showing the total time used for training and retraining SVM classifiers. 

\begin{table}[htbp]
  \centering
  \caption{Sampling and Training Time}
    \begin{tabular}{ccccrrrrr}
    \toprule
    Instance & $K$     & $N$     & Method & {Opt Gap} & Number & Testing-  & Phase-1- & SVM-Training- \\
     & & & & ($\%$) & of Cuts & Time (s) & Time (s) & Time (s)\\
    \midrule
    \multirow{2}[2]{*}{r041} & \multirow{2}[2]{*}{2} & \multirow{2}[2]{*}{80} & BD    & 42.19 & 21520  & $>$ 7200 & -     & - \\
          &       &       & LearnBD & 22.88 & 19761  &  $>$ 7200 & 37.63 & 0.01 \\
\cmidrule{4-9}    \multirow{2}[2]{*}{r046} & \multirow{2}[2]{*}{2} & \multirow{2}[2]{*}{80} & BD    & 1  & 2560  & 62.30 & -     & - \\
          &       &       & LearnBD & 1  & 1708   & 41.70 & 15.29 & 0.01 \\
\cmidrule{4-9}    \multirow{2}[2]{*}{r051} & \multirow{2}[2]{*}{2} & \multirow{2}[2]{*}{100} & BD    &  9.79 & 18500 &  $>$ 7200  & -     & - \\
          &       &       & LearnBD &3.66  & 15210  &  $>$ 7200 & 22.39 & 0.01 \\
\cmidrule{4-9}    \multirow{2}[2]{*}{r054} & \multirow{2}[2]{*}{2} & \multirow{2}[2]{*}{100} & BD    & 1  & 6400  & 585.94  & -     & - \\
          &       &       & LearnBD & 1  & 5398  & 385.89  & 24.41 & 0.38 \\
\cmidrule{4-9}    \multirow{2}[2]{*}{r061} & \multirow{2}[2]{*}{2} & \multirow{2}[2]{*}{100} & BD    & 10.54  & 12100 &  $>$ 7200 & -     & - \\
          &       &       & LearnBD & 9.01  & 12288  &  $>$ 7200 & 28.59 & 0.40 \\
\cmidrule{4-9}    \multirow{2}[2]{*}{r071} & \multirow{2}[2]{*}{2} & \multirow{2}[2]{*}{80} & BD    & 958.69 & 12720 &  $>$ 7200 & -     & - \\
          &       &       & LearnBD & 884.19 & 13194 &  $>$ 7200 & 186.77 & 0.30 \\
\cmidrule{4-9}    \multirow{2}[2]{*}{r075} & \multirow{2}[2]{*}{2} & \multirow{2}[2]{*}{80} & BD    & 10.89  & 9680 &  $>$ 7200 & -     & - \\
          &       &       & LearnBD & 9.24  & 9621  &  $>$ 7200 & 58.94 & 0.22 \\
\cmidrule{4-9}    \multirow{2}[2]{*}{r076} & \multirow{2}[2]{*}{2} & \multirow{2}[2]{*}{80} & BD    &  1  & 3040  & 360.93 & -     & - \\
          &       &       & LearnBD & 1  & 3026  & 276.91 & 51.93 & 0.21 \\
\cmidrule{4-9}    \multirow{2}[2]{*}{r082} & \multirow{2}[2]{*}{2} & \multirow{2}[2]{*}{80} & BD    & 9.77  & 8480  &  $>$ 7200 & -     & - \\
          &       &       & LearnBD &  9.27  &   8873  &  $>$ 7200 & 45.26 & 0.30 \\
    \bottomrule
    \end{tabular}%
  \label{tab:training time}%
\end{table}%

From Table \ref{tab:training time}, the time of Phase 1 is highly dependent on $N$, while the time of each sampling path is the cumulative time for solving $N$ number of RMPs with $n$ cuts in iteration $n$ for $n=1, \ldots, N$. In our computational studies, we use $N=2\times|\Omega|$, i.e., proportional to $|\Omega|$, and thus the time in Phase 1 is highly related to $|\Omega|$. 
 In most of the instances, the time in Phase 1 and the training time of SVM classifiers are relatively short as compared to the total solving time.
Together with the results in Table~\ref{tab:mc-table-precision}, for Instances r051, r061, and r082, LearnBD  reduces the total solving time by $80\%$ to reach a $10\%$ optimality gap even if LearnBD conducts Phase 1 for solving only one testing problem (if we also consider the time in sampling cuts and training classifiers).

\subsection{Classifier Transfer Between Instances}\label{subsec:transfer}

Previously, we propose to solve a given two-stage stochastic optimization problem with cuts sampled from a similar training problem, where the only difference between the training problem and the original problem is the underlying distribution of the uncertain parameter.
Intuitively, if the problem can be solved with training data collected from another training problem in a different size, e.g., the different number of variables and constraints, then the algorithmic efficiency can be further improved because we can use the same training data, or equivalently speaking, we can transfer the classifiers to solve other problems.

In this section, we present the results where the training problem is different from the original problem as an extension.
We consider two instances cap42 and cap62 of CFLP.
We normalize the two cut characteristics, cut violation and number of cuts generated by the same scenario, to eliminate the incompatible effects of the difference between the training and testing problems.
Instead of using the absolute value of cut violation, we scale the cut violation by $ \frac{\sum_{j \in F}\bar{d}_j \bar{c}}{\bar{k}}$, where $\bar{d}_j$ is the nominal value of the uncertain demand of factory $j$, $\bar{c}=\frac{\sum_{i \in W, j \in F} c_{ij}}{|W|\cdot|F|}$ is the average transportation cost, and $\bar{k}=\sum_{i \in W} k_i/|W|$ is the average warehouse setup cost. 
The scalar $ \frac{\sum_{j \in F}\bar{d}_j \bar{c}}{\bar{k}}$ can be reviewed as the relative total transportation cost, which reflects the magnitude of the optimal objective function value of the CFLP problem.
The second cut characteristic, the number of cuts generated by the same scenario, is related to the required number of iterations if using the traditional Benders.
This characteristic is hard to estimate before solving the problem. Thus, we propose to use the size of the transportation network, i.e., $|W|\cdot|F|$, to scale it.

The results are presented in Table~\ref{tab:trasfer}, for which we set the precision $\delta$ as $0.01\%$ and the time limit as one hour.
In the first two rows, we solve cap42 with traditional BD and the proposed LearnBD where the training data is collected from a training problem with the same model parameters (the results are the same as that of cap42 in Table~\ref{tab:fl-table}).
In the third row, we construct a training problem from cap62 and then use the training data to train an SVM classifier and solve cap42.

\begin{table}[htbp]
  \centering
  \caption{Results with Transferred Classifier}
   \resizebox{\textwidth}{!}{
    \begin{tabular}{cccccccccc}
    \toprule
    Inst. & $|\Omega|$ & \multicolumn{1}{c}{Std.Testing} & Method & \multicolumn{1}{c}{Training} & \multicolumn{1}{c}{Std. Training} & Number  & \multicolumn{1}{c}{Opt Gap} & Number & \multicolumn{1}{c}{Total Time} \\
          &       & \multicolumn{1}{c}{ ($ \times$ mean)} &       & \multicolumn{1}{c}{ Instance} & \multicolumn{1}{c}{ ($ \times$ mean)} & \multicolumn{1}{c}{of Iter.} & \multicolumn{1}{c}{($\%$)} & \multicolumn{1}{c}{of Cuts} & \multicolumn{1}{c}{of RMPS (s)} \\
    \midrule
    \multirow{3}[6]{*}{cap42} & \multirow{3}[6]{*}{100} & \multirow{3}[6]{*}{0.1} & BD    & -     &  -      &  30   & 0.01  & 6000  & 111.96\\
\cmidrule{4-10}          &       &       & LearnBD & cap42 & 0.1  &  32   & 0.01  & 4304  & 84.28 \\
\cmidrule{4-10}          &       &       & LearnBD & cap62 & 0.1   &   29  & 0.01  & 5794  & 97.96  \\
    \bottomrule
    \end{tabular}%
    }
  \label{tab:trasfer}%
\end{table}%

In Table~\ref{tab:trasfer}, LearnBD with a transferred SVM classifier, i.e., the third row, also reduces the cumulative time of RMPs as compared to BD. Via comparing the second and the third rows, LearnBD trained with the same instance adds fewer cuts and takes short time to solve RMPs. Therefore, we can conclude that with a proper scaling rule of the two cut features, the training data can also be re-used for solving other instances to improve the solving efficiency.

\section{Conclusions}\label{sec:conclusion} 
In this paper, we developed a learning-enhanced Benders decomposition algorithm to accelerate the solving process of Benders decomposition, one of the most useful algorithms for solving two-stage stochastic programs.
The bottleneck for traditional Benders decomposition is the increasing sizes and the long solving time of RMPs. We restricted RMP sizes over iterations by distinguishing valuable cuts. 
The computational studies based on capacitated facility location and multi-commodity network design instances demonstrated the power of SVM cut classifier. With a proper selection of hyperparameters, the LearnBD algorithm worked efficiently with smaller sizes and  shorter solving time of RMPs compared to traditional Benders decomposition.

 Our numerical results of diverse instances showed that LearBD can achieve better computational performance than BD for solving different types of benchmark two-stage stochastic programs considered in the literature.
We consider the following future research directions to improve our algorithm.
First, we can extend the characteristics and performance indices for the current LearnBD algorithm to capture multiple types of information of cuts. The second direction is to explore the possibility of constructing an online learning algorithm using reinforcement learning, which requires decomposing the effects of multiple cuts added simultaneously into the same RMP. We are also interested in improving LearnBD for solving a broader range of large-scale problems with special structural properties.

\section*{Acknowledgements}
The authors thank the Associate Editor and two reviewers for their constructive feedback and suggestions. The authors gratefully acknowledge the support from the U.S. Department of Engineering (DoE) grant \# DE-SC0018018.

%
%



\begin{thebibliography}{50}
\providecommand{\natexlab}[1]{#1}
\providecommand{\url}[1]{\texttt{#1}}
\expandafter\ifx\csname urlstyle\endcsname\relax
  \providecommand{\doi}[1]{doi: #1}\else
  \providecommand{\doi}{doi: \begingroup \urlstyle{rm}\Url}\fi

\bibitem[Akinc and Khumawala(1977)]{akinc1977efficient}
U.~Akinc and B.~M. Khumawala.
\newblock An efficient branch and bound algorithm for the capacitated warehouse
  location problem.
\newblock \emph{Management Science}, 23\penalty0 (6):\penalty0 585--594, 1977.

\bibitem[Baltean-Lugojan et~al.(2019)Baltean-Lugojan, Bonami, Misener, and
  Tramontani]{sl-baltean2019scoring}
R.~Baltean-Lugojan, P.~Bonami, R.~Misener, and A.~Tramontani.
\newblock Scoring positive semidefinite cutting planes for quadratic
  optimization via trained neural networks.
\newblock 2019.

\bibitem[Beasley(1988)]{beasley1988algorithm}
J.~E. Beasley.
\newblock An algorithm for solving large capacitated warehouse location
  problems.
\newblock \emph{European Journal of Operational Research}, 33\penalty0
  (3):\penalty0 314--325, 1988.

\bibitem[Benders(1962)]{benders1962partitioning}
J.~F. Benders.
\newblock Partitioning procedures for solving mixed-variables programming
  problems.
\newblock \emph{Numerische Mathematik}, 4\penalty0 (1):\penalty0 238--252,
  1962.

\bibitem[Binato et~al.(2001)Binato, Pereira, and Granville]{binato2001new}
S.~Binato, M.~V.~F. Pereira, and S.~Granville.
\newblock A new {B}enders decomposition approach to solve power transmission
  network design problems.
\newblock \emph{IEEE Transactions on Power Systems}, 16\penalty0 (2):\penalty0
  235--240, 2001.

\bibitem[Birge and Louveaux(2011)]{birge2011introduction}
J.~R. Birge and F.~Louveaux.
\newblock \emph{Introduction to Stochastic Programming}.
\newblock Springer, 2011.

\bibitem[Cai et~al.(2001)Cai, McKinney, Lasdon, and Watkins~Jr]{cai2001solving}
X.~Cai, D.~C. McKinney, L.~S. Lasdon, and D.~W. Watkins~Jr.
\newblock Solving large nonconvex water resources management models using
  generalized {B}enders decomposition.
\newblock \emph{Operations Research}, 49\penalty0 (2):\penalty0 235--245, 2001.

\bibitem[Chang and Lin(2011)]{Chang2011LIBSVM}
C.-C. Chang and C.-J. Lin.
\newblock {LIBSVM}: A library for support vector machines.
\newblock \emph{ACM Transactions on Intelligent Systems and Technology (TIST)},
  2\penalty0 (3):\penalty0 1--27, 2011.

\bibitem[Christofides and Beasley(1983)]{christofides1983extensions}
N.~Christofides and J.~E. Beasley.
\newblock Extensions to a lagrangean relaxation approach for the capacitated
  warehouse location problem.
\newblock \emph{European Journal of Operational Research}, 12\penalty0
  (1):\penalty0 19--28, 1983.

\bibitem[Cordeau et~al.(2001)Cordeau, Stojkovi{\'c}, Soumis, and
  Desrosiers]{cordeau2001benders}
J.-F. Cordeau, G.~Stojkovi{\'c}, F.~Soumis, and J.~Desrosiers.
\newblock Benders decomposition for simultaneous aircraft routing and crew
  scheduling.
\newblock \emph{Transportation Science}, 35\penalty0 (4):\penalty0 375--388,
  2001.

\bibitem[Cortes and Vapnik(1995)]{cortes1995support}
C.~Cortes and V.~Vapnik.
\newblock Support-vector networks.
\newblock \emph{Machine Learning}, 20\penalty0 (3):\penalty0 273--297, 1995.

\bibitem[Costa(2005)]{costa2005survey}
A.~M. Costa.
\newblock A survey on {B}enders decomposition applied to fixed-charge network
  design problems.
\newblock \emph{Computers \& Operations Research}, 32\penalty0 (6):\penalty0
  1429--1450, 2005.

\bibitem[Crainic et~al.(2001)Crainic, Frangioni, and
  Gendron]{crainic2001bundle}
T.~G. Crainic, A.~Frangioni, and B.~Gendron.
\newblock Bundle-based relaxation methods for multicommodity capacitated fixed
  charge network design.
\newblock \emph{Discrete Applied Mathematics}, 112\penalty0 (1-3):\penalty0
  73--99, 2001.

\bibitem[Crainic et~al.(2011)Crainic, Fu, Gendreau, Rei, and
  Wallace]{crainic2011progressive}
T.~G. Crainic, X.~Fu, M.~Gendreau, W.~Rei, and S.~W. Wallace.
\newblock Progressive hedging-based metaheuristics for stochastic network
  design.
\newblock \emph{Networks}, 58\penalty0 (2):\penalty0 114--124, 2011.

\bibitem[Crainic et~al.(2014)Crainic, Hewitt, and Rei]{crainic2014partial}
T.~G. Crainic, M.~Hewitt, and W.~Rei.
\newblock \emph{Partial decomposition strategies for two-stage stochastic
  integer programs}.
\newblock CIRRELT, 2014.

\bibitem[Dashti et~al.(2016)Dashti, Conejo, Jiang, and Wang]{dashti2016weekly}
H.~Dashti, A.~J. Conejo, R.~Jiang, and J.~Wang.
\newblock Weekly two-stage robust generation scheduling for hydrothermal power
  systems.
\newblock \emph{IEEE Transactions on Power Systems}, 31\penalty0 (6):\penalty0
  4554--4564, 2016.

\bibitem[Federgruen and Zipkin(1984)]{federgruen1984combined}
A.~Federgruen and P.~Zipkin.
\newblock A combined vehicle routing and inventory allocation problem.
\newblock \emph{Operations Research}, 32\penalty0 (5):\penalty0 1019--1037,
  1984.

\bibitem[Gendron et~al.(2016)Gendron, Scutell{\`a}, Garroppo, Nencioni, and
  Tavanti]{gendron2016branch}
B.~Gendron, M.~G. Scutell{\`a}, R.~G. Garroppo, G.~Nencioni, and L.~Tavanti.
\newblock A branch-and-{B}enders-cut method for nonlinear power design in green
  wireless local area networks.
\newblock \emph{European Journal of Operational Research}, 255\penalty0
  (1):\penalty0 151--162, 2016.

\bibitem[Geoffrion and Graves(1974)]{geoffrion1974multicommodity}
A.~M. Geoffrion and G.~W. Graves.
\newblock Multicommodity distribution system design by {B}enders decomposition.
\newblock \emph{Management Science}, 20\penalty0 (5):\penalty0 822--844, 1974.

\bibitem[He et~al.(2014)He, Daume~III, and Eisner]{he2014learning}
H.~He, H.~Daume~III, and J.~M. Eisner.
\newblock Learning to search in branch and bound algorithms.
\newblock In \emph{Advances in Neural Information Processing Systems (NIPS)},
  pages 3293--3301, 2014.

\bibitem[Holmberg(1990)]{holmberg1990convergence}
K.~Holmberg.
\newblock On the convergence of cross decomposition.
\newblock \emph{Mathematical Programming}, 47\penalty0 (1-3):\penalty0
  269--296, 1990.

\bibitem[Hooker(2007)]{hooker2007planning}
J.~N. Hooker.
\newblock Planning and scheduling by logic-based {B}enders decomposition.
\newblock \emph{Operations Research}, 55\penalty0 (3):\penalty0 588--602, 2007.

\bibitem[Khalil et~al.(2017)Khalil, Dai, Zhang, Dilkina, and
  Song]{khalil2017learning}
E.~Khalil, H.~Dai, Y.~Zhang, B.~Dilkina, and L.~Song.
\newblock Learning combinatorial optimization algorithms over graphs.
\newblock In \emph{Advances in Neural Information Processing Systems}, pages
  6348--6358, 2017.

\bibitem[Khalil et~al.(2016)Khalil, Le~Bodic, Song, Nemhauser, and
  Dilkina]{khalil2016learning}
E.~B. Khalil, P.~Le~Bodic, L.~Song, G.~L. Nemhauser, and B.~N. Dilkina.
\newblock Learning to branch in mixed integer programming.
\newblock In \emph{Proceedings of the thirtieth AAAI Conference on Artificial
  Intelligence (AAAI-16)}, pages 724--731, 2016.

\bibitem[Kleywegt et~al.(2002)Kleywegt, Shapiro, and Homem-de
  Mello]{kleywegt2002sample}
A.~J. Kleywegt, A.~Shapiro, and T.~Homem-de Mello.
\newblock The sample average approximation method for stochastic discrete
  optimization.
\newblock \emph{SIAM Journal on Optimization}, 12\penalty0 (2):\penalty0
  479--502, 2002.

\bibitem[Klibi and Martel(2012)]{klibi2012scenario}
W.~Klibi and A.~Martel.
\newblock Scenario-based supply chain network risk modeling.
\newblock \emph{European Journal of Operational Research}, 223\penalty0
  (3):\penalty0 644--658, 2012.

\bibitem[Klibi et~al.(2010)Klibi, Martel, and Guitouni]{klibi2010design}
W.~Klibi, A.~Martel, and A.~Guitouni.
\newblock The design of robust value-creating supply chain networks: a critical
  review.
\newblock \emph{European Journal of Operational Research}, 203\penalty0
  (2):\penalty0 283--293, 2010.

\bibitem[Kruber et~al.(2017)Kruber, L{\"u}bbecke, and
  Parmentier]{kruber2017learning}
M.~Kruber, M.~E. L{\"u}bbecke, and A.~Parmentier.
\newblock Learning when to use a decomposition.
\newblock In \emph{International Conference on AI and OR Techniques in
  Constraint Programming for Combinatorial Optimization Problems}, pages
  202--210. Springer, 2017.

\bibitem[Laporte et~al.(1994)Laporte, Louveaux, and Mercure]{laporte1994priori}
G.~Laporte, F.~V. Louveaux, and H.~Mercure.
\newblock A priori optimization of the probabilistic traveling salesman
  problem.
\newblock \emph{Operations Research}, 42\penalty0 (3):\penalty0 543--549, 1994.

\bibitem[Magnanti and Wong(1981)]{magnanti1981accelerating}
T.~L. Magnanti and R.~T. Wong.
\newblock Accelerating {B}enders decomposition: {A}lgorithmic enhancement and
  model selection criteria.
\newblock \emph{Operations Research}, 29\penalty0 (3):\penalty0 464--484, 1981.

\bibitem[Melkote and Daskin(2001)]{melkote2001capacitated}
S.~Melkote and M.~S. Daskin.
\newblock Capacitated facility location/network design problems.
\newblock \emph{European journal of operational research}, 129\penalty0
  (3):\penalty0 481--495, 2001.

\bibitem[Minoux(1986)]{minoux1986mathematical}
M.~Minoux.
\newblock \emph{Mathematical Programming: Theory and Algorithms}.
\newblock John Wiley \& Sons, 1986.

\bibitem[Misra et~al.(2018)Misra, Roald, and Ng]{misra2018learning}
S.~Misra, L.~Roald, and Y.~Ng.
\newblock Learning for constrained optimization: Identifying optimal active
  constraint sets.
\newblock \emph{arXiv preprint arXiv:1802.09639}, 2018.

\bibitem[Naoum-Sawaya and Elhedhli(2013)]{naoum2013interior}
J.~Naoum-Sawaya and S.~Elhedhli.
\newblock An interior-point {B}enders based branch-and-cut algorithm for mixed
  integer programs.
\newblock \emph{Annals of Operations Research}, 210\penalty0 (1):\penalty0
  33--55, 2013.

\bibitem[Orchard-Hays et~al.(1968)]{orchard1968advanced}
W.~Orchard-Hays et~al.
\newblock \emph{Advanced Linear-programming Computing Techniques}.
\newblock McGraw-Hill, 1968.

\bibitem[Pomerleau(1991)]{didi-kdd-pomerleau1991efficient}
D.~A. Pomerleau.
\newblock Efficient training of artificial neural networks for autonomous
  navigation.
\newblock \emph{Neural Computation}, 3\penalty0 (1):\penalty0 88--97, 1991.

\bibitem[Rahmaniani et~al.(2017)Rahmaniani, Crainic, Gendreau, and
  Rei]{rahmaniani2017benders}
R.~Rahmaniani, T.~G. Crainic, M.~Gendreau, and W.~Rei.
\newblock The {B}enders decomposition algorithm: {A} literature review.
\newblock \emph{European Journal of Operational Research}, 259\penalty0
  (3):\penalty0 801--817, 2017.

\bibitem[Raidl(2015)]{raidl2015decomposition}
G.~R. Raidl.
\newblock Decomposition based hybrid metaheuristics.
\newblock \emph{European Journal of Operational Research}, 244\penalty0
  (1):\penalty0 66--76, 2015.

\bibitem[Reddy et~al.(2019)Reddy, Dragan, and Levine]{didi-kdd-reddy2019sqil}
S.~Reddy, A.~D. Dragan, and S.~Levine.
\newblock Sqil: Imitation learning via regularized behavioral cloning.
\newblock \emph{arXiv preprint arXiv:1905.11108}, 2019.

\bibitem[Ross et~al.(2011)Ross, Gordon, and
  Bagnell]{didi-kdd-ross2011reduction}
S.~Ross, G.~Gordon, and D.~Bagnell.
\newblock A reduction of imitation learning and structured prediction to
  no-regret online learning.
\newblock In \emph{Proceedings of the fourteenth international conference on
  artificial intelligence and statistics}, pages 627--635, 2011.

\bibitem[Saravanan et~al.(2013)Saravanan, Das, Sikri, and
  Kothari]{saravanan2013solution}
B.~Saravanan, S.~Das, S.~Sikri, and D.~Kothari.
\newblock A solution to the unit commitment problem—a review.
\newblock \emph{Frontiers in Energy}, 7\penalty0 (2):\penalty0 223--236, 2013.

\bibitem[Sch{\"o}lkopf et~al.(2002)Sch{\"o}lkopf, Smola, Bach,
  et~al.]{scholkopf2002learning}
B.~Sch{\"o}lkopf, A.~J. Smola, F.~Bach, et~al.
\newblock \emph{Learning with Kernels: Support Vector Machines, Regularization,
  Optimization, and Beyond}.
\newblock MIT Press, 2002.

\bibitem[Smola(2004)]{Smola2004A}
A.~J. Smola.
\newblock \emph{A Tutorial on Support Vector Regression}.
\newblock Kluwer Academic Publishers, 2004.

\bibitem[Song et~al.(2014)Song, Luedtke, and
  K{\"u}{\c{c}}{\"u}kyavuz]{song2014chance}
Y.~Song, J.~R. Luedtke, and S.~K{\"u}{\c{c}}{\"u}kyavuz.
\newblock Chance-constrained binary packing problems.
\newblock \emph{INFORMS Journal on Computing}, 26\penalty0 (4):\penalty0
  735--747, 2014.

\bibitem[Tang et~al.(2019)Tang, Agrawal, and Faenza]{rl-tang2019reinforcement}
Y.~Tang, S.~Agrawal, and Y.~Faenza.
\newblock Reinforcement learning for integer programming: Learning to cut.
\newblock \emph{arXiv preprint arXiv:1906.04859}, 2019.

\bibitem[Vapnik(1998)]{vapnik1998statistical}
V.~Vapnik.
\newblock \emph{Statistical Learning Theory. 1998}, volume~3.
\newblock Wiley, New York, 1998.

\bibitem[Vapnik(2013)]{vapnik2013nature}
V.~Vapnik.
\newblock \emph{The Nature of Statistical Learning Theory}.
\newblock Springer, 2013.

\bibitem[Vapnik(1999)]{vapnik1999overview}
V.~N. Vapnik.
\newblock An overview of statistical learning theory.
\newblock \emph{IEEE Transactions on Neural Networks}, 10\penalty0
  (5):\penalty0 988--999, 1999.

\bibitem[Wolfe(1970)]{wolfe1970convergence}
P.~Wolfe.
\newblock Convergence theory in nonlinear programming.
\newblock \emph{Integer and Nonlinear Programming}, pages 1--36, 1970.

\bibitem[Zakeri et~al.(2000)Zakeri, Philpott, and Ryan]{zakeri2000inexact}
G.~Zakeri, A.~B. Philpott, and D.~M. Ryan.
\newblock Inexact cuts in {B}enders decomposition.
\newblock \emph{SIAM Journal on Optimization}, 10\penalty0 (3):\penalty0
  643--657, 2000.

\end{thebibliography}


\appendix
\renewcommand{\thetheorem}{\thesection.\arabic{theorem}}
\setcounter{section}{0}
\setcounter{theorem}{0}
\renewcommand{\theequation}{\thesection-\arabic{equation}}
\setcounter{equation}{0}
\setcounter{figure}{0}
\setcounter{table}{0}

\section*{APPENDIX}

\section{Preliminaries of SVM}\label{app:svm}

We start with 
\begin{equation}\label{eq:predict3}
  f_{SVM}(\mathbf{o'}) = \textrm{sign} \bigg[ \mathbf{w}^T \phi(\mathbf{o'}) + b \bigg]  
\end{equation}
and $u =  \mathbf{w}^T \phi(\mathbf{o'}) + b$, where $\phi(\cdot)$ is a unique mapping function such that $K(\mathbf{o}_1, \mathbf{o}_2) = \left< \phi(\mathbf{o}_1), \phi(\mathbf{o}_2) \right>$.
By the kernel trick \citep[see, e.g.,][]{scholkopf2002learning}, we do not have to know the exact form of $\phi(\cdot)$ and we can employ the SVM model only with kernel function $K(\cdot,\cdot)$.
In Proposition~\ref{prop:para}, we show this formulation is equivalent to \eqref{eq:predict2}.
With a penalty hyperparameter $C \geq 0$ assigned to the prediction error, the objective function for solving the parameters is defined as  $\frac{1}{2}  \mathbf{w}^T\mathbf{w}  + C \sum_{d = 1}^{\Gamma} \xi_d$, where the first term representing the flatness and $\xi_d, \  d = 1, \ldots, \Gamma$ is an auxiliary variable for representing loss amount of training data $(\mathbf{o}_d, l_d)$.  
The parameters can be solved by an optimization problem:
\begin{subequations}\label{eq:svmoptprob}
\begin{alignat}{4}
(\textrm{SVM-P}) \ \ \  \mathop{ \textrm{min}}\limits_{\mathbf{w},\xi,b} \ & \ \frac{1}{2}  \mathbf{w}^T\mathbf{w}  + C \sum_{d = 1}^{\Gamma} \xi_d \label{eq:svm-pobj}\\
 \textrm{s.t.} \ & \   l_{d} \cdot \big( \mathbf{w}^T \phi(\mathbf{o}_d) +b \big)  \geq 1  - \xi_{d} \ \  \ \ \ \ {d} = 1, \ldots, \Gamma; \label{subeq:firsttype} \\
& \ \xi_{d} \geq 0 \ \ \  \ \ \ \ \ \ \ \ \ \ \ \ {d} = 1, \ldots, \Gamma ,\label{subeq:secondtype}
\end{alignat}
\end{subequations}
where \eqref{subeq:firsttype} are used to calculate the hinge loss and \eqref{subeq:secondtype} are sign restrictions of $\xi$.
(SVM-P) is a convex optimization problem with convex inequality constraints and a quadratic objective function, and thus it is easy to solve by taking Lagrangian Dual and applying Krash-Kuhn-Tucker (KKT) conditions \citep[see, e.g.,][]{Chang2011LIBSVM}.

\begin{proposition}\label{prop:lower}
The optimal objective value of
\begin{equation}\label{eq:lb}
 \mathop{\textrm{min}}\limits_{\mathbf{w},\mathbf{\xi},b} \ \frac{1}{2}  \mathbf{w}^T\mathbf{w}  + C \sum_{d = 1}^{\Gamma} \xi_{d}
- \sum_{{d}=1}^{\Gamma} a_{d} \left\{ l_{d} \cdot \big( \mathbf{w}^T \phi ( \mathbf{o}_{d}) +b \big) - 1  + \xi_{d}  \right\}
- \sum_{{d}=1}^{\Gamma} v_{d} \xi_{d} 
\end{equation}
with any $\mathbf{a},\mathbf{v} \geq 0$ is a valid lower bound of (SVM-P).
\end{proposition}

\begin{proof}
Assume that $(\mathbf{w}_1,\mathbf{\xi}_1,b_1)$ is an optimal solution to (SVM-P), and therefore it is feasible to the relaxation \eqref{eq:lb}.
Given the constraints in (SVM-P), we have 
$ l_{d} \cdot \big( \mathbf{w}_1^T \phi(\mathbf{o}_{d})+b_1 \big) - 1  + \xi_{1d} \geq 0$ and $ \xi_{1d} \geq 0$ for all $d$. Therefore, for any $\mathbf{a},\mathbf{v} \geq 0$, the objective value of \eqref{eq:lb} based on solution $(\mathbf{w}_1,\mathbf{\xi}_1,b_1)$ is no larger than $\frac{1}{2}\mathbf{w}_1^T\mathbf{w}_1  + C \sum_{d = 1}^{\Gamma} \xi_{1d}$. Moreover, as the optimal objective value of \eqref{eq:lb} is smaller than or equal to the objective value of any feasible solution, we can conclude that the objective value of \eqref{eq:lb} evaluated at the feasible solution $(\mathbf{w}_1,\mathbf{\xi}_1,b_1)$ is always smaller than or equal to $\frac{1}{2}\mathbf{w}_1^T\mathbf{w}_1  + C \sum_{d = 1}^{\Gamma} \xi_{1d}$, which is the optimal objective value of (SVM-P) and thus provides a valid lower bound of (SVM-P). This completes the proof. 
\end{proof}

By associating dual variables $\mathbf{a} \geq 0$ with inequality constraints \eqref{subeq:firsttype} and dual variables $\mathbf{v} \geq 0$ with inequality constraints \eqref{subeq:secondtype}, we can relax those two sets of constraints and then obtain the corresponding Lagrangian function for any feasible solution $(\mathbf{w},\mathbf{\xi},b)$ as 
$$
 L(\mathbf{w},\mathbf{\xi},b;\mathbf{a},\mathbf{v}) = \\
\frac{1}{2}  \mathbf{w}^T\mathbf{w}  + C \sum_{d = 1}^{\Gamma} \xi_{d}
- \sum_{{d}=1}^{\Gamma} a_{d} \left\{ l_{d} \cdot \big( \mathbf{w}^T \phi (\mathbf{o}_{d})+b \big) - 1  + \xi_{d}  \right\}
- \sum_{{d}=1}^{\Gamma} v_{d} \xi_{d} .
$$
By weak duality, the Lagrangian problem
$$
\mathop{\textrm{min}}\limits_{\mathbf{w},\mathbf{\xi},b} L(\mathbf{w},\mathbf{\xi},b;\mathbf{a},\mathbf{v})
$$ yields a valid lower bound of (SVM-P). Moreover, $$
\mathop{\textrm{max}}\limits_{\mathbf{a \geq 0},\mathbf{v \geq 0}} \ \mathop{\textrm{min}}\limits_{\mathbf{w},\mathbf{\xi},b} L(\mathbf{w},\mathbf{\xi},b;\mathbf{a},\mathbf{v})
$$ is the dual problem that seeks the best lower bound.

\begin{definition}
\textit{Krash-Kuhn-Tucker (KKT)} is a set of conditions including: Primal feasibility, dual feasibility, complementary slackness, and the first derivative of Lagrangian function $L(\cdot)$ being zero. If the primal problem is 
$$
\begin{aligned}
\min_x \ & \ f_0(x) \\
\text{subject to} \ & \ f_i(x) \leq 0 \ \ \forall i \in I \\
 \ & \ h_i(x) = 0 \ \ \forall i \in I'
\end{aligned}
$$ and the associated dual multipliers are $\lambda \geq 0$ and $\mu$, then the KKT conditions are:
\begin{itemize}
    \item $f_i(x^*) \leq 0 \ \ \forall i \in I$ and $h_i(x^*) = 0 \ \ \forall i \in I'$ (primal feasibility),
    \item $\lambda^* \geq 0$ (dual feasibility),
    \item $f_i(x^*) \lambda_i^* = 0$ (complementary slackness),
    \item $\nabla f_0(x^*) +  \sum_{i \in I} \lambda_i^* \nabla f_i(x^*) +  \sum_{i \in I'} \mu_i^* \nabla h_i(x^*) = 0$ (first derivative of $L(\cdot)$ is zero).
\end{itemize}
\end{definition}

\begin{proposition}\label{prop:sd&kkt}
The strong duality holds for (SVM-P) and KKT conditions are satisfied at the optimal primal and dual solution pair.
\end{proposition}

\begin{proof}
(SVM-P) has a quadratic objective and affine inequality constraints, and therefore by Slater's condition strong duality holds. Because (SVM-P) is  differentiable, KKT conditions hold at the global optimum. This completes the proof. 
\end{proof}

\begin{theorem}
The optimal objective value of the optimization problem
\begin{equation}\label{eq:ldp}
\mathop{\textrm{max}}\limits_{\mathbf{a \geq 0},\mathbf{v \geq 0}} \ \mathop{\textrm{min}}\limits_{\mathbf{w},\mathbf{\xi},b} \ \frac{1}{2}  \mathbf{w}^T\mathbf{w}  + C \sum_{d = 1}^{\Gamma} \xi_{d}
- \sum_{{d}=1}^{\Gamma} a_{d} \left\{ l_{d} \cdot \big( \mathbf{w}^T \phi(\mathbf{o}_{d})+b \big) - 1  + \xi_{d}  \right\}
- \sum_{{d}=1}^{\Gamma} v_{d} \xi_{d} 
\end{equation}
equals to the optimal objective value of (SVM-P).
\end{theorem}

\begin{proof}
Recall the Lagrangian dual problem 
$$
\mathop{\textrm{max}}\limits_{\mathbf{a \geq 0},\mathbf{v \geq 0}} \ \mathop{\textrm{min}}\limits_{\mathbf{w},\mathbf{\xi},b} \ L(\mathbf{w},\mathbf{\xi},b;\mathbf{a},\mathbf{v}) . 
$$ 
By Proposition \ref{prop:sd&kkt}, strong duality holds and thus the optimal objective value of the dual problem and primal problem are equal.
\end{proof}

\begin{proposition}\label{prop:deri}
The Lagrangian dual function \eqref{eq:lb} in Proposition \ref{prop:lower} can be reformulated as
\begin{subequations}\label{eq:refomlb}
\begin{alignat}{2}
& \ \frac{1}{2} \sum_{{d}=1}^{\Gamma}\sum_{d'=1}^{\Gamma} l_{d} l_{d'} a_{d} a_{d'} K(\mathbf{o}_d,\mathbf{o}_{d'} ) + \sum_{{d}=1}^{\Gamma}a_{d} - \sum_{{d}=1}^{\Gamma} v_{d} \xi_{d}  \\
\text{with} \ & \  \sum_{{d}=1}^{\Gamma} a_{d} l_{d} = 0; \\
& \  C - a_{d} - v_{d} = 0 \ \  \ {d} = 1, \ldots, \Gamma .\ \ \ 
\end{alignat}
\end{subequations}
\end{proposition}

\begin{proof}
The Lagrangian dual function \eqref{eq:lb} is differentiable, and therefore the derivatives associated with $(\mathbf{w},\mathbf{\xi},b)$ at the minimum are equal to zero, i.e.,
\begin{subequations}\label{eq:partial}
\begin{alignat}{2}
\frac{\partial L}{\mathbf{w}  } & = 0 \ \rightarrow \ \mathbf{w} = \sum_{{d}=1}^{\Gamma} a_{d} l_{d} \phi_d \label{eq:avalue}\\
\frac{\partial L}{b } & = 0 \ \rightarrow \ \sum_{{d}=1}^{\Gamma} a_{d} l_{d} = 0\\
\frac{\partial L}{ \xi } & = 0 \ \rightarrow \ C - a_{d} - v_{d} = 0, \ {d} = 1, \ldots, \Gamma \ \rightarrow \ a_{d} \leq c, \ {d} = 1, \ldots, \Gamma.
\end{alignat}
\end{subequations}
Plugging in the results in \eqref{eq:partial}, we can obtain the reformulation of \eqref{eq:lb} in \eqref{eq:refomlb}. This completes our proof. 
\end{proof}

\begin{theorem}
The Lagrangian dual problem \eqref{eq:ldp} is equivalent to solving a convex quadratic program:
\begin{subequations}\label{svm2}
\begin{alignat}{2}
 \mathop{\textrm{max}}\limits_{\mathbf{a}} \ & \ \frac{1}{2} \sum_{{d}=1}^{\Gamma}\sum_{d'=1}^{\Gamma} l_{d} l_{d'} a_{d} a_{d'}K(\mathbf{o}_d,\mathbf{o}_{d'} ) + \sum_{{d}=1}^{\Gamma}a_{d} \\
 \textrm{s.t.} \ & \  \sum_{{d}=1}^{\Gamma} a_{d} l_{d} = 0; \\
& \ 0 \leq a_{d} \leq C \ \  \ {d} = 1, \ldots, \Gamma . \ \ \
\end{alignat}
\end{subequations}
\end{theorem}

\begin{proof}
By Proposition \ref{prop:deri}, we obtain an equivalent formulation of \eqref{eq:ldp} as follows.
\begin{subequations}\label{eq:ldpreform1}
\begin{alignat}{2}
\mathop{\textrm{max}}\limits_{\mathbf{a}, \mathbf{v } \geq 0} \ & \ \frac{1}{2} \sum_{{d}=1}^{\Gamma}\sum_{d'=1}^{\Gamma} l_{d} l_{d'} a_{d} a_{d'}K(\mathbf{o}_d,\mathbf{o}_{d'} ) + \sum_{{d}=1}^{\Gamma}a_{d} - \sum_{{d}=1}^{\Gamma} v_{d} \xi_{d} \label{eq:ldpreform1obj}  \\
\text{with} \ & \  \sum_{{d}=1}^{\Gamma} a_{d} l_{d} = 0; \\
& \  C - a_{d} - v_{d} = 0 \ \  \ {d} = 1, \ldots, \Gamma ;\label{eq:ldpreform1c}  \ \ \ 
\end{alignat}
\end{subequations}

In the third term in the objective function \eqref{eq:ldpreform1obj}, all $v_d \xi_d , \ \forall d = 1, \ldots, \Gamma$ are zero at the optimum because of the complementary slackness by Proposition \ref{prop:sd&kkt}. 
Therefore, we can discard the third term without loss of optimality.
Moreover, because $v_{d} \geq 0, \ \forall d = 1, \ldots, \Gamma$, we can combine \eqref{eq:ldpreform1c} with $\mathbf{v } \geq 0$ and derive valid constraints $a_{d} \leq C, \ \forall d = 1, \ldots, \Gamma  $, which helps to eliminate variables $v_d , \ \forall d = 1, \ldots, \Gamma$.
Finally, we can rewrite model \eqref{eq:ldpreform1} as shown in \eqref{svm2} \citep[see, e.g.,][]{Chang2011LIBSVM}. This completes our proof. 
\end{proof}

\begin{proposition}\label{prop:para}
The parameter of the classifier in \eqref{eq:predict3} are $\mathbf{w}^* = \sum_{{d}=1}^{\Gamma} a_{d}^* l_{d} \phi_d $ and $b^* = 1 - \sum_{d'=1}^{\Gamma} l_{d'} ( a^*_{d'} K(\mathbf{o}_{d}, \mathbf{o}_{d'}) )$ for any $d = 1, \ldots, \Gamma$ associated with $a_{d}^* \in (0,C)$. 
The three prediction functions \eqref{eq:predict}, \eqref{eq:predict2} and \eqref{eq:predict3} are equivalent to each other,
where the support vector set $S$ in \eqref{eq:predict2} contains all $(\mathbf{o}_d, l_d), \ d = 1, \ldots, \Gamma$ such that $a^*_d > 0$.
\end{proposition}

\begin{proof}
The value of $\mathbf{w}^*$ is obtained by \eqref{eq:avalue} and it shows the equivalence between \eqref{eq:predict} and \eqref{eq:predict3}.
Assume that we solve and obtain an optimal solution $\mathbf{a}^*$ to \eqref{svm2}.
Then following the complementary slackness:
$$
 a_d \left[ l_{d} \cdot \big( \mathbf{w}^T \phi(\mathbf{o}_{d}) + b\big)  - 1  + \xi_{d} \right]= 0 , \ v_d  \xi_d = 0, \ \forall d = 1, \ldots, \Gamma,
$$
we have:
\begin{itemize}
    \item If $a^*_d = C > 0$, then $ l_{d} \cdot \big(  \mathbf{w^*}^T \phi(\mathbf{o}_{d} ) + b\big)  = 1  - \xi^*_{d} $. By $a^*_d = C - v^*_d$ we have $v^*_d = 0$ and thus $\xi^*_d \geq 0$. The observation $d$ is called non-margin support vector.
    \item If $0 < a^*_d < C$, then $l_{d} \cdot \big( \mathbf{w^*}^T \phi(\mathbf{o}_{d}) + b\big)  = 1  - \xi^*_{d} $. Similarly, we have $v^*_d > 0$ and thus $\xi^*_d = 0$. The observation $d$ is called margin support vector. Therefore, we can compute $b^* = 1 - \sum_{d'=1}^{\Gamma} l_{d'} ( a^*_{d'} K(\mathbf{o}_{d}, \mathbf{o}_{d'}))$ with any $d = 1, \ldots, \Gamma$ associated with $a_{d}^* \in (0,C)$, 
    \item If $ a^*_d =0$, then this type of observation $d$ does not affect the value of the second prediction function. Therefore, we can build a support vector set $S$ of $(\mathbf{o}_d, l_d), \ d = 1, \ldots, \Gamma$ with $a^*_d \neq 0$ and thus simplify \eqref{eq:predict} as \eqref{eq:predict2}. 
\end{itemize}
\end{proof}

\section{Detailed Formulations of Problems for Computational Studies}

\subsection{CFLP}\label{app:cflp}

Consider a set $W$ of production plants (facilities) and a set $F$ of factories which have uncertain demand $\tilde{d}$.
The setup cost of facility $i, \ \forall i \in W$ is $k_i$ and the production capacity limit is $u_i$.
The demand of factory $j, \forall j \in F$ is uncertain and can be satisfied by products produced in facility $i, \ \forall i \in W$ if it is open with a unit transportation cost $c_{ij}$, and the unmet demand will generate lost-sale with a unit penalty cost $\rho_j$.
One needs to decide a subset of facilities to open before the realization of the demand to minimize the expected total cost.

The two-stage stochastic programming model consists of two types of decisions. 
We define first-stage binary decision variables $x_i , \ \forall i \in W$ such that $x_i =1$ if we open facility $i$ and $x_i=0$ otherwise. 
In the second stage, we obtain the demand value from each factory and define continuous decision variables $y_{ij} \geq 0, \ \forall i \in W, \ j \in F$, which represent transportation units from facility $i$ to factory $j$.
The model aims to find the best decisions to minimize the facility setup cost, expected transportation cost, and expected lost-sale cost.
The first-stage formulation is:
\begin{equation} \label{eq:facility1st}
\begin{aligned}
(\textrm{CFLP}) \ \ \  \min_{x} \ & \  \sum_{i \in W}k_i x_i +  \sum_{\omega \in \Omega} p_{\omega} Q_{\omega}(x) \\   
 \ \textrm{s.t.} \ &  \ x_i \in \{ 0,1\} \ \ \  \ i \in W  . \ \ \\
\end{aligned}
\end{equation}
The second-stage problem for each scenario $\omega$ is defined using variables $y_{ij}, \ i \in W , \ j \in F$ and auxiliary variables $\alpha_j, \ j \in F$ that denote the amount of unmet demand. We have
\begin{equation} \label{eq:facility2nd}
\begin{aligned}
Q_{\omega}(x) \ =  \ \  \min_{y, \alpha } \ & \ \sum_{i \in W} \sum_{ j \in F} c_{ij} y_{ij} + \sum_{j \in F} \rho_j \alpha_j   \\   
 \ \textrm{s.t.} \ &  \ \sum_{j \in F}  y_{ij} \leq u_i x_i  \ \ \ \ \ \ \ \ \ \ \  i \in W ; \\
 & \   \tilde{d}_{\omega,j} - \sum_{i \in W} y_{ij}  \leq \alpha_{j}  \ \ \ \ \  j \in F; \\
 & \ y_{ij} \geq 0 \  \ \ \ \ \ \ \ \ \ \ \ \ \ \ \ \ \ \ \  i \in W, \ j \in F; \ \ \\
 & \ \alpha_j  \geq 0 \ \ \ \ \ \ \ \ \ \ \ \ \ \ \ \ \ \ \ \  \ j \in F . \ \  \ \ 
\\
\end{aligned}
\end{equation}
By allowing unmet demand, the problem always has a feasible solution and Benders decomposition only generates optimality cuts.
Then, we derive the dual of second-stage problems and formulate SPs as shown in Section~\ref{subsec:bd}. 
By defining dual variables $h_i, \ \forall i\in W$ and $\pi_j , \ \forall j \in F$, respectively associated with the first and second constraints in model~\eqref{eq:facility2nd}, we formulate the subproblem in scenario $\omega$ as
\begin{equation}\label{eq:facilitySP}
\begin{aligned}
(\textrm{SP}_{\omega})  \qquad  \max_{h, \pi} \ & \ - \sum_{i \in W} u_i x_i h_i +   \sum_{j \in F} \tilde{d}_{\omega,j}  \pi_j   \\
 \textrm{s.t.} \ \ & - h_i - \pi_j  \leq c_{ij} \ \  i \in W, \ j \in F; \\
& 0 \leq \pi_{j} \leq \rho \ \  \ \ \ \  \ \  j \in F ;\ \\
& \ h_i \geq 0 \ \  \ \ \ \  \ \  \  \ \ \ \  i \in W .\\ 
\end{aligned}
\end{equation}
Letting $V^{\omega,t}$ be a collection of extreme points of $\textrm{SP}_{\omega}$ that have been identified when reaching iteration $t$, we formulate
\begin{equation}\label{eq:facilityRMP}
\begin{aligned}
(\textrm{RMP}^t) \ \ \    \min_{x,\theta} \ & \ \sum_{i \in W} k_i x_i  + \sum_{\omega \in \Omega} p_{\omega} \theta_\omega \\   
 \textrm{s.t.} \ &  \ \theta_\omega \geq - \sum_{i \in W} u_i x_i \hat{h_i} +   \sum_{j \in F} \tilde{d}_{\omega,j}\hat{\pi_j}  \ \ \ \   (\hat{h}_{ i} , \hat{\pi}_{j}  ) \in V^{\omega,t}, \  \omega \in \Omega;\\
& \ x_i \in \{ 0,1\} \ \ \ \ \  \ \ \ \ \  \ \ \ \ \ \ \ \ \ \ \ \ \ \ \ \ \ \ \  i \in W . \ \ \ \ \   \ \
\\
\end{aligned}
\end{equation}

\subsection{CMND}\label{app:cmnd}

Consider a directed network with node set $N$, arc set $A$, and commodity set $K$. An uncertain $\tilde{v}_k$ amount of commodity $k, \forall k \in K$ must be routed from an origin node, $o_k \in N$, to a destination node, $d_k \in N$. 
The installation cost and arc capacity of arc $(i,j),\ \forall (i,j) \in A$ are $f_{ij}$ and $u_{ij}$, respectively.
The cost for transporting one unit of commodity $k, \ \forall k \in K$ on installed arc $(i,j),  \ \forall (i,j) \in A$ is $c_{ij}^k$. 
One needs to decide a subset of arcs to install before the realization of the demand to minimize the expected total cost.
In the first stage, we make binary decisions $x_{ij}, \ \forall (i,j) \in A$ such that $x_{ij} = 1$ if we install arc $(i,j)$. 
In the second stage, we obtain the demand of each commodity and then solve non-negative continuous decisions $y_{ij}^k, \forall (i,j) \in A, \ k \in K$, which represents transportation units of commodity $k$ on arc $(i,j)$.

The first-stage formulation is:
\begin{equation} \label{eq:mc-1}
\begin{aligned}
(\textrm{CMND}) \ \ \  \min_{x} \ & \  \sum_{(i,j) \in A}f_{ij} x_{ij} +  \sum_{\omega \in \Omega} p_{\omega} Q_{\omega}(x) \\   
 \ \textrm{s.t.} \ &  \ x_{ij} \in \{ 0,1\} \ \ \  \ (i,j) \in A  . \ \ \\
\end{aligned}
\end{equation}
The second-stage problem for each scenario $\omega$ is defined with decision variables $y_{ij}^k, \ \forall (i ,j) \in A, \ k \in K$ and auxiliary variables $\alpha_i^k, \ \forall i \in N, \ k \in K$ for denoting unmet demand:
\begin{equation} \label{eq:mc-2}
\begin{aligned}
Q_{\omega}(x) \ =  \ \  \min_{y, \alpha } \ & \ \sum_{(i,j) \in A} \left[ \sum_{ k \in K} c_{ij}^k y_{ij}^k + B \alpha_i^k  \right] \\   
 \ \textrm{s.t.} \ &  \ \sum_{j: (j,i) \in A}  y_{ji}^k - \sum_{j: (i,j) \in A}  y_{ij}^k \leq \tilde{d}_i^k + \alpha_i^k \ \ \ \ \ \ \ \ \ \ \ & i \in N, \ k \in K ; \\
 & \   \sum_{k \in K} y_{ij}^k \leq u_{ij}x_{ij}    \ \ \ \ \   & (i,j) \in A; \\
 & \ y_{ij}^k \geq 0 \  \ \ \ \ \ \ \ \ \ \ \ \ \ \ \ \ \ \ \  &  (i ,j) \in A, \ k \in K; \\
 & \ \alpha_i^k  \geq 0 \ \ \ \ \ \ \ \ \ \ \ \ \ \ \ \ \ \ \ \  \  & i \in N, \ k \in K . 
\\
\end{aligned}
\end{equation}
Define an auxiliary demand unmet cost $B$.
The parameter $\tilde{d}_i^k$ is set to $\tilde{v}^k$ if node $i$ is the origin of the commodity $k$,  $-\tilde{v}^k$ is node $i$ is the destination of the commodity $k$, or $0$ otherwise.

By allowing unmet demand, the problem always has a feasible solution and Benders decomposition only generates optimality cuts.
We derive the dual of second-stage problems and formulate SPs as shown in Section~\ref{subsec:bd}. 
By defining dual variables $h_i^k, \ \forall i\in N, \ k \in K$ and $\pi_{ij} , \ \forall (i,j) \in A$, respectively associated with the first and second constraints in model~\eqref{eq:mc-2}, we formulate the subproblem in scenario $\omega$ as
\begin{equation}\label{eq:mc-sp}
\begin{aligned}
(\textrm{SP}_{\omega})  \qquad  \max_{h, \pi} \ & \  \sum_{i \in N,  k \in K}  - \tilde{d}_i^k h_i^k  -   \sum_{(i,j) \in A} u_{ij} x_{ij}  \pi_{ij}   \\
 \textrm{s.t.} \ \ & h_i^k - h_j^k - \pi_{ij}   \leq c_{ij}^k \ \  & (i,j) \in A, \ k \in K; \\
& \ 0 \leq \pi_{ij}  \ \  \ \ \ \  \ \  \  \ \ \ \ & (i,j) \in A.\\ 
& \ 0 \leq h_i^k \leq B \ \  \ \ \ \  \ \  \  \ \ \ \ & i \in N, \ k \in K .
\end{aligned}
\end{equation}
Letting $V^{\omega,t}$ be a collection of extreme points of $\textrm{SP}_{\omega}$ that have been identified when reaching iteration $t$, we formulate
\begin{equation}\label{eq:mc-RMP}
\begin{aligned}
(\textrm{RMP}^t) \ \ \    \min_{x,\theta} \ & \  \sum_{(i,j) \in A}f_{ij} x_{ij}    + \sum_{\omega \in \Omega} p_{\omega} \theta_\omega \\   
 \textrm{s.t.} \ &  \
 \theta_\omega \geq   \sum_{i \in N,  k \in K}  \tilde{d}_i^k \hat{h}_i^k  -   \sum_{(i,j) \in A} u_{ij} x_{ij}  \hat{\pi}_{ij}   \ \ \ \ &  (\hat{h}_{ i} , \hat{\pi}_{ij}  ) \in V^{\omega,t}, \  \omega \in \Omega;\\
& \ \ x_{ij} \in \{ 0,1\} \ \ \  \ & (i,j) \in A .
\end{aligned}
\end{equation}

\end{document}